\documentclass[11pt]{amsart}

\usepackage{amsmath,amssymb,amsthm}
\usepackage[foot]{amsaddr}
\usepackage{bm}
\usepackage{mathrsfs}
\usepackage[all]{xy}
\usepackage{booktabs}
\usepackage{fullpage}
\usepackage{color}
\usepackage{hyperref}
\usepackage{mathtools}
\usepackage[abbrev,msc-links]{amsrefs}
\hypersetup{
 hidelinks,
 bookmarksopen=true,
}

\newcommand{\cx}{\mathbb{C}}

\newcommand{\ai}{\sqrt{-1}}

\newcommand{\isom}{\xrightarrow{\sim}}

\newcommand{\del}{{\partial}}
\newcommand{\delbar}{{\overline{\partial}}}

\newcommand{\ve}{\varepsilon}

\theoremstyle{plain}
\newtheorem{theorem}{Theorem}[section]

\newtheorem{lemma}[theorem]{Lemma}
\newtheorem{proposition}[theorem]{Proposition}
\newtheorem{corollary}[theorem]{Corollary}
\theoremstyle{definition}
\newtheorem{definition}[theorem]{Definition}

\theoremstyle{definition}
\newtheorem{remark}[theorem]{Remark}

\newtheorem{question}[theorem]{Question}

\begin{document}

\title[Ueda's lemma via uniform H\"ormander estimates for flat line bundles]{Ueda's lemma via uniform H\"ormander estimates for flat line bundles}
\author{Yoshinori Hashimoto}
\email{yhashimoto@omu.ac.jp}
\author{Takayuki Koike}
\email{tkoike@omu.ac.jp}
\address{Department of Mathematics, Osaka Metropolitan University, 3-3-138, Sugimoto, Sumiyoshi-ku, Osaka, 558-8585, Japan.}
\date{\today}

\maketitle

{\bf Abstract.} 
We establish H\"ormander-type $L^2$-estimates for the $\delbar$-operators that hold uniformly for all nontrivial flat holomorphic line bundles on compact K\"ahler manifolds. Our result can be regarded as a $\delbar$-version of Ueda's lemma on the operator norm of \v{C}ech coboundaries for flat line bundles and indeed recovers the original version of Ueda's lemma for compact K\"ahler manifolds. A partial generalisation for $(p,0)$-forms on Ricci-flat manifolds is also given.
\vskip3mm

\tableofcontents

\section{Introduction}
The existence theorem of solutions to the $\delbar$-equations with $L^2$-estimates play an essential role in complex analytic geometry. 
Let $X$ be a complex manifold of dimension $n$ and $F$ be a holomorphic vector bundle on $X$. 
In this paper, our interest is in the case where $X$ is compact K\"ahler and $F$ is a line bundle. In what follows, we fix a K\"ahler metric $g$ of $X$ and a Hermitian (fibre) metric $h$ of $F$. 
In order to find a solution $u$ of the $\delbar$-equation $\delbar u = v$ with $L^2$-estimate for a given smooth $\delbar$-closed $(p, q)$-form $v$ with values in $F$, 
according to the argument originating in \cite{Hormander}, 
it is known to be important to show an estimate such as 
\begin{equation}\label{ineq:laplacian}
\int_X|w|_{h, g}^2\,dV_g \leq C\cdot \int_X\langle\Delta w, w\rangle_{h, g}\,dV_g\quad (w\in A^{p,q}(X,F))
\end{equation}
with a constant $C>0$ independent of $w$, where $dV_g$ is the volume form of $X$ defined by $g$ and $\Delta := \delbar\,\delbar^*_h+\delbar^*_h\delbar$ is the (complex) Laplace--Beltrami operator acting on the space $A^{p,q}(X,F)$ of all smooth $(p, q)$-forms with values in $F$. 

One of the most important ideas to obtain inequalities as above is to apply Nakano's identity. When the Chern curvature tensor $\sqrt{-1}\Theta_h$ of $h$ is semi-positive, in order to show the inequality of type (\ref{ineq:laplacian}) in accordance with this idea, one needs to assume that $(\sqrt{-1}\Theta_h)_x$ has at least $2n+1-p-q$ positive eigenvalues on each $x\in X$ (\cite{AN}, \cite{AV}, \cite{Girbau}, see also \cite[\S 4.D]{Demailly}). 
Thus, it is difficult to show the inequality of type (\ref{ineq:laplacian}) for a line bundle $F$ with semi-positively curved metric $h$ at least by the arguments based on Nakano's identity when $p+q$ is small, especially when $p<n$ and $q=1$. 

Nevertheless, even when $h$ is {\rm flat} (i.e.~$\sqrt{-1}\Theta_h\equiv 0$), 
sometimes one can solve $\delbar$-equations on $F$ (consider a holomorphically non-trivial line bundle of degree $0$ on an elliptic curve, for example). 
In such cases, Ueda's lemma on $L^\infty$-estimate of the \v{C}ech coboundary operator $\delta$ (\cite[Lemma 4]{Ueda}, see also Corollary \ref{cor:main} below) can be regarded as the counterpart of H\"ormander-type estimate for the operator $\delbar$ via the \v{C}ech--Dolbeault correspondence (see \S \ref{section:C-D_corresp}). 
Motivated by this observation, we investigate an $L^2$-variant of Ueda's lemma in the present paper. 

Denote by $\mathcal{P}(X)$ the set of all the flat holomorphic line bundles. 
When $X$ is compact K\"ahler, $\mathcal{P}(X)$ can be regarded as a subset of the Picard group ${\rm Pic}(X)$ of $X$. 
As a subset of ${\rm Pic}(X)$, $\mathcal{P}(X)$ consists of finitely many connected components, and each of the components is homeomorphic to the connected component ${\rm Pic}^0(X)$ which contains the holomorphically trivial line bundle $\mathbb{I}_X$ (see \S \ref{section:flathlb}). 
Denote by $\mathsf{d}_0$ the Euclidean distance on ${\rm Pic}^0(X)$ ($\mathsf{d}_0$ is defined by choosing the Euclidean coordinates and hence is not canonical, c.f.~Remark \ref{rmk:amb_for_d}). 
As explained in \S \ref{section:ext_of_invariant_dist}, one can extend $\mathsf{d}_0$ to define an invariant distance (in the sense of Ueda \cite[\S 4.1]{Ueda}) on $\mathcal{P}(X)$, which will be denoted by $\mathsf{d}$. 
Our main results can be stated as follows. 

\begin{theorem}\label{thm:main_1}
Let $(X, g)$ be a compact K\"ahler manifold. 
There exists a constant $K>0$ such that, 
for any element $F\in \mathcal{P}(X)\setminus\{\mathbb{I}_X\}$ and any smooth $\delbar$-closed $(0, 1)$-form $v$ with values in $F$ whose Dolbeault cohomology class $[v]\in H^{0, 1}(X, F)$ is trivial, there exists a unique smooth global section $u$ of $F$ such that $\delbar u = v$ and 
\[
\sqrt{\int_X|u|_{h}^2\,dV_g} \leq \frac{K}{\mathsf{d}(\mathbb{I}_X, F)} \sqrt{\int_X|v|_{h, g}^2\,dV_g}
\]
hold for a flat metric $h$ on $F$. 
\end{theorem}

A feature of the above theorem is that it applies \textit{uniformly} to all $F\in \mathcal{P}(X)\setminus\{\mathbb{I}_X\}$, with the factor $1/\mathsf{d}(\mathbb{I}_X, F)$ blowing up as $F$ approaches $\mathbb{I}_X$; the appearance of such a distance is a feature that do not commonly arise in the case of positive curvature, which is the usual setting for H\"ormander's estimates. We may call the above result a ``$\delbar$-version'' of Ueda's lemma for flat holomorphic line bundles, in which $\mathsf{d}(\mathbb{I}_X, F)$ also appear in a similar manner. Before discussing connections to Ueda's lemma in more details, we present a generalisation of the above result to $(p,0)$-forms. 

\begin{theorem}\label{thm:main_2}
Let $(X, g)$ be a K\"ahler manifold of dimension $n$ 
and $p$ be a non-negative integer which is less than or equal to $n$. 
Assume that $g$ is Ricci-flat. Then there exist a neighbourhood $B$ of $\mathbb{I}_X$ in $\mathcal{P}(X)$ and a constant $K>0$ such that, 
for any element $F\in B\setminus\{\mathbb{I}_X\}$ and any smooth $\delbar$-closed $(p, 1)$-form $v$ with values in $F$ whose Dolbeault cohomology class $[v]\in H^{p, 1}(X, F)$ is trivial, there exists a unique smooth $(p, 0)$-form $u$ with values in $F$ such that $\delbar u = v$ and 
\[
\sqrt{\int_X|u|_{h, g}^2\,dV_g} \leq \frac{K}{\mathsf{d}(\mathbb{I}_X, F)} \sqrt{\int_X|v|_{h, g}^2\,dV_g}
\]
hold for a flat metric $h$ on $F$. In particular, 
\begin{equation*}
	\ker (\delbar \colon A^{p,0}(X,F) \to A^{p,1}(X,F)) = 0 
\end{equation*}
and hence
\begin{equation*}
	H^{p,0} (X,F) = 0
\end{equation*}
for all $F\in B\setminus\{\mathbb{I}_X\}$. 
\end{theorem}

Note that, as is clear from our proof of this theorem, Theorem \ref{thm:main_2} remains true even after replacing $B$ with ${\rm Pic}^0(X)$ if the following condition holds: $H^{p, 0}(X, F)=0$ holds for any $F\in {\rm Pic}^0(X)\setminus \{\mathbb{I}_X\}$ (see Remark \ref{rmk:prof_of_thm_main_2}). This condition always holds when $p=0$ by Lemma \ref{lem:flat_H0=0}; for $p>0$ see \S \ref{eg:p>0}. 

By combining Theorem \ref{thm:main_1} and a standard argument on the \v{C}ech--Dolbeault correspondence (see \S \ref{section:C-D_corresp}), one has an alternative proof of Ueda's lemma for compact K\"ahler manifolds.

\begin{corollary}[=Ueda's lemma {\cite[Lemma 4]{Ueda}}]\label{cor:main}
Let $X$ be a compact K\"ahler manifold and $\{U_j\}$ be an open covering of $X$ such that $\#\{U_j\}<\infty$, each $U_j$ is Stein, and $\{U_j\}$ trivialises any $F\in\mathcal{P}(X)$. 
Then there exists a constant $K>0$ such that, for any $F\in\mathcal{P}(X)$ and any \v{C}ech $0$-cochain $\mathfrak{f}:=\{(U_j, f_{j})\}\in \check{C}^0(\{U_j\}, \mathcal{O}_X(F))$, the inequality
\[
\mathsf{d}(\mathbb{I}_X, F)\cdot \max_j\sup_{U_j}|f_j|_h\leq K\cdot \max_{j, k}\sup_{U_{jk}}|f_{jk}|_h
\]
holds, where $\{(U_{jk}, f_{jk})\}:=\delta \mathfrak{f}\in \check{C}^1(\{U_j\}, \mathcal{O}_X(F))$ is the \v{C}ech coboundary of $\mathfrak{f}$. 
\end{corollary}

Note that $X$ need not be K\"ahler in Ueda's original proof, and that the open cover $\{ U_j \}$ satisfying all the hypotheses above does exist by Lemma \ref{lmufmcv}. 
See also \cite[Appendix A]{GS} for another generalisation of Ueda's lemma, in which estimates for the operator norm of the \v{C}ech coboundary operators with respect to suitable $L^2$-norms are given for a family of open coverings of a manifold. 

In order to prove Theorem \ref{thm:main_1} and \ref{thm:main_2}, we show the inequality (\ref{ineq:laplacian}) for $(p, 0)$-forms with a suitable constant $C$ (when $F\in\mathcal{P}(X)\setminus\{\mathbb{I}_X\}$ is sufficiently close to $\mathbb{I}_X$), since 
\[
\int_X|\delbar u|_{h, g}^2\,dV_g = \int_X\langle \Delta u, u\rangle_{h, g}\,dV_g
\]
holds for $(p, 0)$-forms. 
For this purpose, we construct a suitable $C^\infty$ global frame $\sigma\colon X\to F$ with $|\sigma|_h\equiv 1$ and consider the correspondence
\[
A^{p, q}(X,F)\ni u\mapsto \widehat{u}:=\frac{u}{\sigma}\in A^{p, q} (X), 
\]
where $A^{p, q} (X)$ denotes the space of all the smooth $(p, q)$-forms on $X$. 
Instead of studying the operator $\delbar$ on $F$ directly, we define a {\it perturbed $\delbar$-operator} by considering the conjugate of the operator $\delbar\colon A^{p, 0}(X,F)\to A^{p, 1}(X,F)$ by this correspondence and evaluate the minimum eigenvalue of the corresponding (Witten-type, see e.g.~\cite{Witten}) perturbed Laplace operator. 

\vskip3mm
{\bf Organisation of the paper. }
We recall some background materials and establish various preliminary results in \S \ref{section:prelim}. We start by recalling some well-known facts for flat holomorphic line bundles over compact K\"ahler manifolds in \S \ref{section:flathlb}, 
the distance $\mathsf{d}$ on $\mathcal{P}(X)$ in \S \ref{section:ext_of_invariant_dist}, 
the \v{C}ech--Dolbeault correspondence in \S \ref{section:C-D_corresp}, and the Albanese varieties in \S \ref{section:alb}. The materials in \S \ref{section:elrtori} are elementary but necessary for the results in \S \ref{section:dboflbpdb}; we believe that many results therein are new, although some of them may well be known to the experts or implicit in the literature. The end product of the results in \S \ref{section:elrtori} and \S \ref{section:dboflbpdb} is the perturbed $\delbar$-operators, which turn out to be very important in analysing the $\delbar$-operators on flat line bundles that are ``close'' to the trivial bundle. Elementary analytic properties of the perturbed $\delbar$-operators are established in \S \ref{scgaopdop} by applying well-known techniques in geometric analysis and PDE theory. All these results are applied to prove Theorems \ref{thm:main_1}, \ref{thm:main_2}, and Corollary \ref{cor:main} in \S \ref{section:proofmr}. Finally, in \S \ref{section:excex}, we recover Theorem \ref{thm:main_1} for elliptic curves by a more concrete argument, and provide counterexamples to naive generalisations of Theorem \ref{thm:main_2}. 
\vskip3mm
{\bf Notation. }
We fix our notational conventions as follows. 
\begin{itemize}
	\item $A^{p,q} (X)$ (resp.~$A^{p,q} (X,F)$) denotes the set of smooth differential $(p,q)$-forms (resp.~$F$-valued smooth $(p,q)$-forms) on $X$.
	\item $T_X$ (resp.~$T_X^*$) denotes the holomorphic tangent (resp.~cotangent) sheaf of $X$, which we identify with the holomorphic vector bundle $T^{1,0} X$ (resp.~$(T^*X)^{1,0}$). We write $\Omega_X^p := \bigwedge^p T^*_X$ for the sheaf of holomorphic $p$-forms on $X$.
	\item $\omega$ stands for a fixed K\"ahler form on $X$, which we identify with the (fixed) K\"ahler metric $g$. For $\alpha, \beta \in A^{p,q} (X)$, we write $\langle \alpha , \beta \rangle_{\omega} (x)$ for the value of the $\omega$-metric inner product of $\alpha, \beta$ at $x \in X$. We may also write $\langle \alpha , \beta \rangle_{\omega}$ when the point $x \in X$ is obvious from the context. The (pointwise) norm is defined by $\Vert \alpha \Vert_{\omega}^2 := \langle \alpha , \alpha \rangle_{\omega}$. The $L^2$-inner product is defined, as usual, by $(\alpha , \beta )_{L^2} := \int_X \langle \alpha , \beta \rangle_{\omega} \,dV_\omega= \int_X \langle \alpha , \beta \rangle_{\omega} \omega^n/n!$.
	\item Throughout the paper, we work with a fixed basis for $H^0 (X , \Omega_X^1)$ which we assume is orthonormal with respect to the above $L^2$-inner product; see \S \ref{section:alb} for more details.
	\item $\Delta:= \delbar \delbar^* + \delbar^* \delbar$ stands for the $\delbar$-Laplacian with respect to $\omega$. The set of harmonic $(p,q)$-forms is denoted by $\mathbb{H}^{p,q}$.
	\item $\mathcal{P}(X)$ denotes the set of all flat holomorphic line bundles on $X$. We write $\mathbb{I}_X \in \mathcal{P}(X)$ for the holomorphically trivial bundle.
	\item ${\rm Pic}^0(X)$ denotes the set of all topologically trivial holomorphic line bundles on $X$. 
\end{itemize}
\vskip3mm
{\bf Acknowledgement. }
The first author is supported by JSPS KAKENHI Grant Number 19K14524. 
The second author is supported by JSPS KAKENHI Grant Number 20K14313.


\section{Preliminaries}\label{section:prelim}

\subsection{Flat holomorphic line bundles}\label{section:flathlb}

Let $X$ be a compact K\"ahler manifold; many results in this section hold more generally for a compact complex manifold, but in this paper we focus on K\"ahler manifolds; see \cite[Sections 1.2 and 1.4]{Kob} for more details. We say that a holomorphic line bundle $F$ on $X$ is {\rm flat} if $F$ admits a Hermitian metric $h$ whose Chern curvature is identically zero. This is equivalent to $F$ being in the image of the morphism $i\colon H^1(X, {\rm U}(1))\to {\rm Pic}(X)=H^1(X, \mathcal{O}_X^*)$ which is induced by the natural injection ${\rm U}(1)\to \mathcal{O}_X^*$ from the sheaf of locally constant ${\rm U}(1)$-valued functions to the sheaf of nowhere vanishing holomorphic functions (${\rm U}(1):=\{t\in\mathbb{C}\mid |t|=1\}$); this condition is equivalent to the existence of an open covering $\{U_j\}$ of $X$ and a local trivialisation $e_j$ of $F$ on each $U_j$ such that $|e_j/e_k|\equiv 1$ holds on each $U_j\cap U_k$, and hence in turn equivalent to the vanishing of the Chern curvature of $h$ (defined by $|e_j|_h\equiv 1$) by \cite[Proposition 1.4.21]{Kob} (noting that the flat $h$-connection $D$ therein can be taken to be the Chern connection when $F$ is holomorphic; see the proof of \cite[Proposition 1.2.5]{Kob}).

\begin{definition}
	We write $\mathcal{P}(X)$ for the set of all equivalence classes of flat holomorphic line bundles over $X$, where the equivalence relation is given by isomorphisms of holomorphic line bundles.
\end{definition}

When $X$ is compact K\"ahler, it is known that the map $i\colon H^1(X, {\rm U}(1))\to H^1(X, \mathcal{O}_X^*)$ is injective and that the image of $i$ coincides with the kernel of the first Chern-class map $c_1\colon H^1(X, \mathcal{O}_X^*)\to H^2(X, \mathbb{R})$ (a theorem of Kashiwara, see \cite[\S 1]{Ueda} for example), which implies that one can naturally identify $\mathcal{P}(X)$ with $H^1(X, {\rm U}(1))$. 
The following lemma explains the relation between $\mathcal{P}(X)$ and the set ${\rm Pic}^0(X)$ of all equivalence classes of topologically trivial holomorphic line bundles over $X$, where the equivalence relation is again given by isomorphisms of holomorphic line bundles. 
\begin{lemma}\label{lem:kahler_pic0_p_relation}
Assume that $X$ is compact K\"ahler. 
Then, as a subset of ${\rm Pic}(X)$, 
${\rm Pic}^0(X)$ is the connected component of $\mathcal{P}(X)$ which contains $\mathbb{I}_X$. 
Moreover, 
$\mathcal{P}(X)$ consists of finitely many connected components, and each component is homeomorphic to ${\rm Pic}^0(X)$. 
The set $\mathcal{P}(X)$ coincides with ${\rm Pic}^0(X)$ when $H^2(X, \mathbb{Z})$ is torsion-free. 
\end{lemma}

In this lemma and henceforth, we regard ${\rm Pic}(X)$ as a topological group. 
The topology of ${\rm Pic}(X)$ is defined by considering the Euclidean topology for ${\rm Pic}^0(X)$ and the discrete topology for the N\'eron--Severi group ${\rm NS}(X)$ of $X$ (recall the standard fact that ${\rm Pic}^0(X)$ is isomorphic to the $2d$-dimensional torus $({\rm U}(1))^{2d}$, where $d = {\rm dim}H^1(X, \mathcal{O}_X)$, and the exact sequence $0\to {\rm Pic}^0(X)\to {\rm Pic}(X)\to {\rm NS}(X)\to 0$ of groups). 

\begin{proof}[Proof of Lemma \ref{lem:kahler_pic0_p_relation}]
As is clear from the commutativity of the following diagram, the equations ${\rm Pic}^0(X) = {\rm Ker}\,c_1^\mathbb{Z}$ and $\mathcal{P}(X) = {\rm Ker}\,c_1$ imply that ${\rm Pic}^0(X)$ is a topological subgroup of $\mathcal{P}(X)$, where $c_1^\mathbb{Z}\colon {\rm Pic}(X)\to H^2(X, \mathbb{Z})$ is the first Chern-class map (here we employ an unconventional notation for denoting the first Chern-class map in order to avoid the confusion with the map $c_1\colon {\rm Pic}(X)\to H^2(X, \mathbb{R})$ above). 
\begin{displaymath}
\xymatrixcolsep{4pc}\xymatrixrowsep{4pc}\xymatrix{
{\rm Pic}(X) \ar@{->}[r]^-{c_1^{\mathbb{Z}}}  \ar@{->}[rrd]_-{c_1} & {\rm NS}(X) \ar@{->}[r]^-{\rm inclusion} & H^2(X, \mathbb{Z}) \ar@{->}[d]^-{\rm natural}  \\
& & H^2(X, \mathbb{R}) 
} 
\end{displaymath}

${\rm Pic}^0(X)$ is clearly path connected since it is isomorphic to $({\rm U}(1))^{2d}$ endowed with the Euclidean topology, and hence it is the identity component of $\mathcal{P} (X)$ by recalling the topology of ${\rm Pic} (X)$ as defined above.

As ${\rm NS}(X)$ is a finitely generated abelian group, it follows from the structure theorem that there exists a finite abelian group $G$ such that ${\rm NS}(X)\cong \mathbb{Z}^r\oplus G$, where $r$ is a non-negative integer (the Picard number of $X$). 
Note that the kernel of the natural map ${\rm NS}(X)\to H^2(X, \mathbb{R})$ coincides with $G\subset {\rm NS}(X)$, since $G$ is the torsion subgroup of ${\rm NS}(X)$. 
Thus one has the exact sequence 
\begin{equation}\label{eq:exact_seq_pic0_pic_g}
0\to {\rm Pic}^0(X)\to \mathcal{P}(X)\to G\to 0
\end{equation} 
of groups, from which the lemma follows (note that $G=0$ if $H^2(X, \mathbb{Z})$ is torsion-free). 
\end{proof}

Note that the above proof shows that we have 
\begin{equation} \label{eqpxdcpdcc}
	\mathcal{P}(X) = {\rm Pic}^0(X) \oplus \left( \bigoplus_{t \in G\setminus\{0\}} {\rm Pic}^t (X) \right)
\end{equation}
where we identify $G$ with a finite set in $\mathcal{P}(X)$ by choosing a representative in the preimage of the map $\mathcal{P}(X)\to G\to 0$ in (\ref{eq:exact_seq_pic0_pic_g}) and write
\begin{equation*}
	 {\rm Pic}^t (X) :=  t \oplus {\rm Pic}^0(X)
\end{equation*}
for each $t \in G \subset \mathcal{P} (X)$.

The following lemma is fundamental.
\begin{lemma}\label{lem:flat_H0=0}
For $F\in \mathcal{P}(X)$, 
\[
H^0(X, F) = \begin{cases}
\mathbb{C} & \text{if $F$ is holomorphically trivial,} \\
0 & \text{otherwise} 
\end{cases}
\]
holds. 
\end{lemma}

\begin{proof}
Take a flat metric $h$ of $F$. 
Then, for a global holomorphic section $f$ of $F$, the function $\log |f|_h$ is plurisubharmonic, from which the assertion holds by applying the maximum principle. 
\end{proof}

A similar argument also proves the following lemma, which allows us to discuss flat holomorphic line bundles without a specific reference to flat metrics; in particular, the choice of the flat metric $h$ in Theorems \ref{thm:main_1}, \ref{thm:main_2}, and Corollary \ref{cor:main} is of no significance. 

\begin{lemma} \label{lmuqflmfbf}
	If $h_1$ and $h_2$ are two flat Hermitian metrics on a holomorphic line bundle $F$ over a compact K\"ahler manifold, there exists $c \in \mathbb{R}$ such that $e^c h_1 = h_2$.
\end{lemma}

\begin{proof}
	We observe that $\log(h_1 / h_2)$ is a nowhere vanishing pluriharmonic function on a compact complex manifold, and hence must be a global real constant.
\end{proof}



When $F$ is holomorphic we have a canonical holomorphic structure $\delbar_F$ (resp.~$\delbar_{\mathbb{I}_X}$) on $F$ (resp.~$\mathbb{I}_X$) given by the $\delbar$-operator of the underlying complex manifold. We observe the following fact.

\begin{lemma} \label{lmhlstfdbf}
Suppose that $F$ is a $C^{\infty}$-complex line bundle over a compact K\"ahler manifold $X$. $F$ is trivial as a $C^{\infty}$-complex line bundle if and only if there exists a $C^{\infty}$-bundle isomorphism $f \colon F \isom \mathbb{I}_X$.

Moreover, the following hold if $F$ is a flat holomorphic line bundle.
\begin{enumerate}
	\item $F$ is holomorphically trivial if and only if there exists a $C^{\infty}$-bundle isomorphism $f \colon F \isom \mathbb{I}_X$ such that $\delbar_{\mathbb{I}_X} = (f^{-1})^* \circ \delbar_F \circ f^*$.
	\item If $F$ is holomorphically non-trivial, $(f^{-1})^* \circ \delbar_F \circ f^*$ defines an integrable partial connection which does not equal $\delbar_{\mathbb{I}_X}$, and hence a non-canonical holomorphic structure, on $\mathbb{I}_X$.
\end{enumerate}

\end{lemma}

See \cite[Theorem 2.1.53]{DK} for more details on the partial connections and holomorphic structures. 

\subsection{Extension of the Euclidean distance on ${\rm Pic}^0(X)$ to $\mathcal{P}(X)$}\label{section:ext_of_invariant_dist}

As is mentioned in \S 1, we attach the Euclidean distance $\mathsf{d}_0$ to the torus ${\rm Pic}^0(X)$, and extend it to define a distance on $\mathcal{P}(X)$. 
Note that $\mathsf{d}_0$ is an invariant distance in the sense of Ueda \cite[\S 4.1]{Ueda}. 
In this section, we show the following lemma in order to ensure the existence of an extension of $\mathsf{d}_0$ to $\mathcal{P}(X)$ which is also an invariant distance. 

\begin{lemma}
	Suppose that $\mathsf{d}_0$ is an invariant distance on ${\rm Pic}^0(X)$ with respect to which ${\rm Pic}^0(X)$ is compact. Then there exists an invariant distance $\mathsf{d}$ on $\mathcal{P}(X)$ such that its restriction to ${\rm Pic}^0(X)$ agrees with $\mathsf{d}_0$ and $\mathsf{d}$ makes $\mathcal{P}(X)$ compact.
\end{lemma}

\begin{proof}
	We recall $\mathcal{P}(X) = \bigoplus_{t \in G} {\rm Pic}^t(X)$ from (\ref{eqpxdcpdcc}), where $G$ is a finite subset of $\mathcal{P}(X)$ corresponding to the first Chern classes that are torsion. A distance $\mathsf{d}_0$ on ${\rm Pic}^0(X)$ naturally induces one, say $\mathsf{d}_t$, on ${\rm Pic}^t(X)$ for each $t \in G\setminus\{0\}$ by
	\begin{equation*}
		\mathsf{d}_t ( t+ v_1 , t+v_2) := \mathsf{d}_0 (v_1 , v_2),
	\end{equation*}
	where $v_1 , v_2 \in {\rm Pic}^0(X)$, recalling that ${\rm Pic}^t(X) = t \oplus {\rm Pic}^0(X)$ is a direct sum.
	
	Setting $\delta := \mathrm{diam}_{\mathsf{d}_0} ({\rm Pic}^0(X)) >0$, we immediately find $\delta = \mathrm{diam}_{\mathsf{d}_t} ({\rm Pic}^t(X))$ for all $t \in G$. We then define a map $\widetilde{\mathsf{d}} \colon \mathcal{P} (X) \times \mathcal{P} (X) \to \mathbb{R}$ as follows: for each $x \in {\rm Pic}^t(X)$ and $y \in {\rm Pic}^s(X)$, $t,s \in G$, we define
\begin{equation*}
	\widetilde{\mathsf{d}} (x,y) :=
	\begin{cases}
		\mathsf{d}_t (x,y) \quad &(t=s) \\
		\mathsf{d}_t (x,t ) + \mathsf{d}_s (s,y) + \delta \quad &(t \neq s)
	\end{cases}
\end{equation*}
where $\mathsf{d}_0 (x,0)$ is meant to stand for $\mathsf{d}_0 (x , \mathbb{I}_X)$. We then define
\begin{equation*}
	\mathsf{d} (x,y) := \inf_{z \in \mathcal{P} (X)} \widetilde{\mathsf{d}} (x -z,y-z)
\end{equation*}
and prove that it defines an invariant distance on $\mathcal{P}(X)$.

We first prove the invariance. For $x = t+v_1 \in {\rm Pic}^t(X)$ and $y = s+v_2 \in {\rm Pic}^s(X)$ (recalling again the direct sum ${\rm Pic}^t(X) = t \oplus {\rm Pic}^0(X)$), we note
\begin{align*}
	\tilde{\mathsf{d}} (t+v_1 , s+v_2) &= \mathsf{d}_t (t+v_1, t) + \mathsf{d}_s (s+v_2,s) + \delta \\
	&= \mathsf{d}_0 (v_1, 0) + \mathsf{d}_0 (v_2,0) + \delta  \\
	&= \mathsf{d}_0 (-v_1, 0) + \mathsf{d}_0 (-v_2,0) + \delta  \\
	&= \tilde{\mathsf{d}} (-t-v_1 , -s-v_2),
\end{align*}
by the invariance of $\mathsf{d}_0$, and
\begin{equation*}
	\mathsf{d} (x,y) = \mathsf{d} (x-z,y-z)
\end{equation*}
for any $z \in \mathcal{P}(X)$ by definition. These properties immediately imply that $\mathsf{d}$ is invariant in the sense of Ueda \cite[\S 4.1]{Ueda}.

It remains to show that $\mathsf{d}$ is a distance on $\mathcal{P}(X)$. We first observe the following: suppose $x,y \in {\rm Pic}^t(X)$ for some $t \in G$. Then, writing $x= t + v_1$ and $y= t+v_2$, we find
\begin{align*}
	\mathsf{d} (x,y) &= \inf_{s \in G , w \in {\rm Pic}^0(X)} \mathsf{d}_{t-s} (t+v_1 - s-w,t+v_2- s-w) \\
	&= \inf_{w \in {\rm Pic}^0(X)} \mathsf{d}_{0} (v_1 -w,v_2-w) \\
	&= \mathsf{d}_{0} (v_1,v_2)
\end{align*}
by the invariance of $\mathsf{d}_0$. If $x = t+v_1 \in {\rm Pic}^t(X)$ and $y =s+v_2 \in {\rm Pic}^s(X)$ for some $t,s \in G$ with $t \neq s$, we find
\begin{align*}
	\mathsf{d} (x,y) &= \inf_{u \in G , w \in {\rm Pic}^0(X)} \widetilde{\mathsf{d}} (t+v_1 - u-w,s+v_2- u-w) \\
	&= \inf_{u \in G , w \in {\rm Pic}^0(X)}  \left( \mathsf{d}_{t-u} (t+v_1 - u-w,t-u ) + \mathsf{d}_{s-u} (s-u, s+v_2- u-w) \right) + \delta \\
	&=\inf_{w \in {\rm Pic}^0(X)} \left( \mathsf{d}_{0} (v_1 , w ) + \mathsf{d}_{0} (v_2,w) \right) + \delta 
\end{align*}
Note that all the infima above are in fact minima by the compactness of ${\rm Pic}^0(X)$ with respect to $\mathsf{d}_0$.

The argument above implies that the only nontrivial part is the triangle inequality, which we prove now. For $x = t+v_1 \in {\rm Pic}^t(X)$, $y=s+v_2 \in {\rm Pic}^s(X)$, $z=u+v_3 \in {\rm Pic}^u(X)$, we first consider the case when $s$, $t$, $u$ are all pairwise distinct. We pick $w' \in {\rm Pic}^0(X)$ such that $\inf_{w \in {\rm Pic}^0(X)} \left( \mathsf{d}_{0} (v_1 , w ) + \mathsf{d}_{0} (v_2,w) \right) = \mathsf{d}_{0} (v_1 , w') + \mathsf{d}_{0} (v_2,w')$, and compute
\begin{align*}
	\mathsf{d} (x,z) &= \inf_{w \in {\rm Pic}^0(X)} \left( \mathsf{d}_{0} (v_1 , w ) + \mathsf{d}_{0} (v_3,w) \right) + \delta  \\
	&\le \delta  + \mathsf{d}_{0} (v_1 , w') + \mathsf{d}_{0} (v_3,w')  \\
	&\le \delta  + \mathsf{d}_{0} (v_1 , w' )  + \delta  + \mathsf{d}_{0} (v_2,w') + \inf_{w'' \in {\rm Pic}^0(X)} \left( \mathsf{d}_{0} (v_2 , w'') + \mathsf{d}_{0} (v_3,w'') \right)  \\
	&= \inf_{w \in {\rm Pic}^0(X)} \left( \mathsf{d}_{0} (v_1 , w ) + \mathsf{d}_{0} (v_2,w) \right) + \inf_{w'' \in {\rm Pic}^0(X)} \left( \mathsf{d}_{0} (v_2 , w'' ) + \mathsf{d}_{0} (v_3,w'') \right) + 2\delta \\
	&= \mathsf{d} (x,y) + \mathsf{d} (y,z).
\end{align*}
If $t=u$ and $t \neq s$, we have $\mathsf{d} (x,z) \le \delta < \mathsf{d} (x,y) + \mathsf{d} (y,z)$. If $t \neq u$ and $t=s$, we have
\begin{align*}
	\mathsf{d} (x,z) &= \inf_{w \in {\rm Pic}^0(X)} \left( \mathsf{d}_{0} (v_1 , w ) + \mathsf{d}_{0} (v_3,w) \right) + \delta  \\
	&\le \inf_{w \in {\rm Pic}^0(X)} \left(  \mathsf{d}_{0} (v_1 , v_2 ) + \mathsf{d}_{0} (v_2 , w ) + \mathsf{d}_{0} (v_3,w) \right) + \delta  \\
	&=\mathsf{d} (x,y) + \mathsf{d} (y,z),
\end{align*}
and similarly for $t \neq u$ and $s=u$. The case when $t=s=u$ is obvious.

Thus $\mathsf{d}$ defines an invariant distance on $\mathcal{P}(X)$, and it is immediate that its restriction to ${\rm Pic}^0(X)$ is $\mathsf{d}_0$ and that $\mathcal{P}(X)$ is compact with respect to $\mathsf{d}$.
\end{proof}

\subsection{\v{C}ech--Dolbeault correspondence}\label{section:C-D_corresp}

In this subsection, $X$ is a compact complex manifold and $F$ is a holomorphic line bundle on $X$. 
Fix a Hermitian metric $g$ of $X$, a Hermitian (fibre) metric $h$ on $F$, and denote by $\omega$ the $(1,1)$-form associated to $g$. In what follows we often identify $g$ with $\omega$, especially when $g$ is K\"ahler. 

We further fix an open covering $\{U_j\}$ of X such that $\#\{U_j\}<\infty$ and each $U_j$ is Stein and contractible, such that $\{U_j\}$ trivialises any flat line bundle $F \in \mathcal{P} (X)$. The existence of such an open cover can be proved as follows, which should be surely well-known to the experts but we spell out the proof for completeness.

\begin{lemma} \label{lmufmcv}
	Let $p_1 : \tilde{X} \to X$ be the universal cover of $X$ and let $\pi_1 (X,*)$ be the fundamental group of $X$ with the basepoint $*$. Let $\{U'_j\}$ be a finite open cover of $X$ such that $p_1^{-1}(U'_j) = \bigsqcup_{g \in \pi_1 (X,*)} V_{j,g}$ and $p_1 |_{V_{j,g}} \colon V_{j,g}\to U'_j$ is a homeomorphism for all $j$. There exists a finite (contractible) Stein open cover $\{U_j\}$ which is a refinement of $\{U'_j\}$ and trivialises all $F \in \mathcal{P} (X)$.
\end{lemma}

\begin{proof}
	It suffices to show that $\{U'_j\}$ trivialises all $F \in \mathcal{P} (X)$, by replacing each $U'_j$ with polydisks inside. We recall that for any $F \in \mathcal{P}(X)$ there exists a representation $\rho : \pi_1 (X, *) \to {\rm U(1)}$ such that $F = (\tilde{X} \times \cx / \sim ) =: \tilde{X} \times_{\rho} \cx$, where the equivalence class $\sim$ is defined by the $\pi_1 (X, *)$-action
	\begin{equation*}
		\tilde{X} \times \cx \ni (x,l) \mapsto (g \cdot x , \rho(g)l) \in \tilde{X} \times \cx
	\end{equation*}
	defined by $g \in \pi_1 (X, *)$; see \cite[Proposition 1.4.21]{Kob}. The bundle projection map $p_2 \colon \tilde{X} \times_{\rho} \cx \to X$ is defined by $p_2 := p_1 \circ \mathrm{pr}_1$, where $\mathrm{pr}_1$ is the first projection (which descends to the map on the quotient after composing with $p_1$). We then find, as required,
	\begin{equation*}
		p_2^{-1} (U'_j) = \left[ \mathrm{pr}^{-1}_1 \left( \bigsqcup_{g \in \pi_1 (X,*)} V_{j,g} \right) \right] =  \left[  \bigsqcup_{g \in \pi_1 (X,*)} V_{j,g} \times \cx \right] \cong U'_j \times \cx
	\end{equation*}
	where $[ \cdot ]$ is the equivalence class under $\sim$.
\end{proof}

In this section, we fix an open cover as above once and for all. We use the following two terms in order to make the arguments clear:

\begin{itemize}
\item For a positive constant $M$, we say that the condition {\bf ($L^2$-estimate)$_M$} holds for $(X, F, \omega, h)$ if, for any smooth $\delbar$-closed $(0, 1)$-form $v$ with values in $F$ whose Dolbeault cohomology class $[v]\in H^{0, 1}(X, F)$ is trivial, there exists a smooth global section $u$ of $F$ such that $\delbar u=v$ and 
\[
\sqrt{\int_X|u|_h^2\,dV_\omega} \leq M \sqrt{\int_X|v|_{h, \omega}^2\,dV_\omega}
\]
hold. \qed
\item For a positive constant $K$, we say that the condition {\bf ($L^\infty$-estimate)$_K$} holds for $(X, F, h, \{U_j\})$ if, for any \v{C}ech $1$-coboundary $\{(U_{jk}, f_{jk})\}\in \check{Z}^1(\{U_j\}, \mathcal{O}_X(F))$ whose \v{C}ech cohomology class $[\{(U_{jk}, f_{jk})\}]\in \check{H}^1(\{U_j\}, \mathcal{O}_X(F))$ is trivial, there exists a \v{C}ech $0$-cochain $\{(U_j, f_{j})\}\in \check{C}^0(\{U_j\}, \mathcal{O}_X(F))$ such that $\delta \{(U_{j}, f_{j})\}=\{(U_{jk}, f_{jk})\}$ and 
\[
\max_j\sup_{U_j}|f_j|_h\leq K\cdot \max_{j, k}\sup_{U_{jk}}|f_{jk}|_h
\]
hold. \qed
\end{itemize}

By using these terminologies, Theorem \ref{thm:main_1} can be reworded as follows: when $\omega$ is a K\"ahler form, there exists a constant $K>0$ such that the condition {\bf ($L^2$-estimate)$_{K/\mathsf{d}(\mathbb{I}_X, F)}$} holds for any $(X, F, \omega, h)$ with $F\in \mathcal{P}(X)\setminus \{\mathbb{I}_X\}$ and flat $h$. 
On the other hand, Ueda's lemma can be reworded as follows: 
there exists a constant $K>0$ such that the condition {\bf ($L^\infty$-estimate)$_{K/\mathsf{d}(\mathbb{I}_X, F)}$} holds for any $(X, F, h, \{U_j\})$ with $F\in \mathcal{P}(X)\setminus \{\mathbb{I}_X\}$ and flat $h$ 
(note that the solution $\{(U_{j}, f_{j})\}$ of the $\delta$-equation is unique when $h$ if flat, since $H^0(X, F)=0$ holds in such case by Lemma \ref{lem:flat_H0=0}). 

As our motivation comes from Ueda's lemma, here let us compare these two conditions in order to study the relation between Theorem \ref{thm:main_1} and Ueda's lemma.  By applying the standard argument on the \v{C}ech--Dolbeault correspondence, we obtain the following: 
\begin{lemma}\label{lem:CD-corresp}
Assume that the condition {\bf ($L^2$-estimate)$_M$} holds for $(X, F, \omega, h)$ and a constant $M>0$. 
Then there exist positive constants $C_1=C_1(\{U_j\}, h, \omega)$ and $C_2=C_2(\{U_j\}, h, \omega)$ such that the condition {\bf ($L^\infty$-estimate)$_{C_1M+C_2}$} holds for $(X, F, h, \{U_j\})$. 
\end{lemma}

\begin{proof}
Fix a partition of unity $\{(U_j, \rho_j)\}$ which is subordinate to the open cover $\{U_j\}$, 
and a relatively compact open subsets $U_j^*\Subset U_j$ such that $\{U_j^*\}$ is also an open covering of $X$. 

Take a \v{C}ech $1$-coboundary $\{(U_{jk}, f_{jk})\}\in \check{Z}^1(\{U_j\}, \mathcal{O}_X(F))$ whose \v{C}ech cohomology class is trivial. 
Then it is well-known that, for a smooth section $g_j := \sum_{\ell\not=j}\rho_\ell f_{j\ell}$ of $F|_{U_j}$, $\{(U_j, \delbar g_j)\}$ patches to define a global smooth $(0, 1)$-form $v$ with values in $F$, whose Dolbeault class is the one which corresponds to the \v{C}ech cohomology class $[\{(U_{jk}, f_{jk})\}]$ via the \v{C}ech--Dolbeault correspondence (therefore the Dolbeault cohomology class $[v]$ is trivial). By using the solution $u$ of the equation $\delbar u=v$, one can construct the $\delta$-primitive $\{(U_j, f_j)\}$ by letting $f_j:=g_j-u$. 

In what follows, we estimate the $L^\infty$-norms of $f_j$'s by using $A:= \max_{j, k}\sup_{U_{jk}}|f_{jk}|_h$. As it follows from $|g_j|_h\leq \sum_{\ell\not=j}\rho_\ell |f_{j\ell}|_h\leq A \sum_{\ell\not=j}\rho_\ell \leq A$ that 
\[
\int_{U_j}|g_j|_h^2\,dV_\omega \leq A^2\cdot {\rm Vol}_\omega(U_j)\quad \left({\rm Vol}_\omega(U_j):= \int_{U_j}\,dV_\omega\right), 
\]
and from the condition {\bf ($L^2$-estimate)$_M$} and the inequality 
$|v|_{h, \omega}(x)\leq \sum_{\ell\not=j}\|\delbar \rho_\ell\|_\omega(x)\cdot |f_{j\ell}(x)|_h
\leq A\cdot \sup_{x'\in X}\sum_{\ell}\|\delbar \rho_\ell\|_\omega(x')$ for each $x\in X$ that 
\[
\int_{U_j}|u|_h^2\,dV_\omega 
\leq \int_X|u|_h^2\,dV_\omega 
\leq M^2\cdot A^2\cdot \left(\sup_{x\in X}\sum_{\ell}\|\delbar \rho_\ell\|_\omega(x)\right)^2\cdot {\rm Vol}_\omega(U_j), 
\]
one has that 
\begin{equation}\label{eq:1_in_prf_CD_corresp}
\sqrt{\int_{U_j}|f_j|_h^2\,dV_\omega} \leq A\cdot (C_1'M+C_2')
\end{equation}
holds for constants $C_1'$ and $C_2'$ which depend only on $\{(U_j, \rho_j)\}$ and $\omega$. 
The mean value inequality for the holomorphic function $f_j/e_j$, where $e_j$ is a local holomorphic frame of $F$ on $U_j$, implies that the inequality 
\begin{equation}\label{eq:2_in_prf_CD_corresp}
\sup_{U_j^*}|f_j|_h \leq C_3'\sqrt{\int_{U_j}|f_j|_h^2\,dV_\omega}
\end{equation}
holds for a constant $C_3'$ which only depends on $\{U_j\}$, $\{U_j^*\}$, $\inf_{U_j}|e_j|_h$, and $\sup_{U_j}|e_j|_h$ (strictly speaking, one needs to slightly modify this part when the open covering $\{U_j\}$ is not fine enough; see Remark \ref{rmk:not_suff_fine_case} below). 
By combining the inequalities (\ref{eq:1_in_prf_CD_corresp}) and (\ref{eq:2_in_prf_CD_corresp}), one has that 
$\sup_{U_j^*}|f_j|_h \leq A\cdot (C_1M+C_2'')$ holds for constants $C_1$ and $C_2''$ which depend only on $\{(U_j, \rho_j)\}$, $\{e_j\}$, $\{U_j^*\}$, $\omega$, and $h$. 

For each point $x\in U_j$, one can estimate $|f_j(x)|_h$ by using $k$ such that $x\in U_k^*$ as follows: 
$|f_j(x)|_h\leq |f_k(x)|_h+|f_{jk}(x)|_h\leq A\cdot (C_1M+C_2'') + A= A\cdot (C_1M+C_2''+1)$. 
Thus the assertion holds by letting $C_2:=C_2''+1$. 
\end{proof}

\begin{remark}\label{rmk:not_suff_fine_case}
When the open covering $\{U_j\}$ is not fine enough, 
it may possible that $\sup_{U_j}|e_j|_h=\infty$ holds for some $j$, where we are using the notation in the proof of Lemma \ref{lem:CD-corresp}. 
In such a case, we need to take an open set $V_j$ such that $U_j^*\Subset V_j\Subset U_j$. 
Then, as it is clear that $\sup_{V_j}|e_j|_h<\infty$, we can recover the inequality (\ref{eq:2_in_prf_CD_corresp}) by applying the mean value inequality for the holomorphic function $f_j/e_j$ on a disc centred at each point of $U_j^*$ with radius sufficiently smaller than the width of the set $V_j\setminus U_j^*$ (instead of $U_j\setminus U_j^*$). 
\end{remark}

It seems natural to pose the following: 
\begin{question}
Assume that the condition {\bf ($L^\infty$-estimate)$_K$} holds for $(X, F, h, \{U_j\})$ and a constant $K>0$. 
Does the condition {\bf ($L^2$-estimate)$_M$} hold for $(X, F, \omega, h)$ for some constant $M$? If so, what is the relation between $K$ and $M$?
\end{question}

Note that, in order to answer this question in accordance with the standard argument of the \v{C}ech--Dolbeault correspondence, first one needs to obtain the solution $u_j\colon U_j\to F|_{U_j}$ of the $\delbar$-equation $\del u_j = v|_{U_j}$ on each $U_j$ for a given $(0, 1)$-from $v$ with estimate of $\sup_{U_j}|u_j|_h$ by using the $L^2$-norm of $v$ (then one can apply {\bf ($L^\infty$-estimate)$_K$} for $\{(U_{jk}, u_j-u_k)\}$ to obtain a $\delta$-primitive $\{(U_j, f_j)\}$ with $L^\infty$-estimate). By using this solution, one can construct the $\delta$-primitive $u$ of $v$ by letting $u:=u_j-f_j$ on each $U_j$), which seems to be difficult. 

\subsection{Albanese varieties}\label{section:alb}

In what follows, we assume that $X$ is a compact K\"ahler manifold with a K\"ahler metric $\omega$. We fix a basis $\{ \zeta_j \}_{j=1}^d$ for the $\mathbb{C}$-vector space $H^0(X, \Omega_X^1)$ once and for all, and assume that it is orthonormal with respect to the $L^2$-inner product defined by $\omega$, where $d := \dim H^0(X, \Omega_X^1)$. We also fix the generators $[\gamma_1], [\gamma_2], \dots, [\gamma_{2d}]$ of the torsion-free part
\begin{equation*}
	H_1(X, \mathbb{Z})_{\rm free} := H_1(X, \mathbb{Z}) / \mathrm{torsion}
\end{equation*}
of the first homology group. 

The Albanese variety 
\begin{equation} \label{eqdfalbv}
\mathrm{Alb} (X) := H^0 (X, \Omega_X^1)^{\vee} / H_1 (X, \mathbb{Z})_{\rm free},	
\end{equation}
is a torus associated to $X$. Note that we also have the Albanese map ${\rm alb} \colon X \to  \mathrm{Alb} (X)$, defined with respect to a fixed basepoint $* \in X$, as
\begin{equation*}
	{\rm alb} (x) := \left[ H^0 (X, \Omega_X^1) \ni \zeta \mapsto \int_*^x \zeta \in \mathbb{C} \right]
\end{equation*}
where $[ \cdot ]$ stands for the equivalence class defined by integration along cycles in $H_1 (X, \mathbb{Z})_{\rm free}$. 
Note also that the homomorphism $H_1(X, \mathbb{Z})\ni [\gamma]\mapsto (\zeta\mapsto \int_\gamma \zeta)\in H^0(X, \Omega_X^1)^\vee$ maps any torsion element of $H_1(X, \mathbb{Z})$ to zero and induces an injection $H_1(X, \mathbb{Z})_{\rm free}\to H^0(X, \Omega_X^1)^\vee$, by which $H_1(X, \mathbb{Z})_{\rm free}$ is naturally regarded as a subgroup of $H^0 (X, \Omega_X^1)^{\vee}$ in the definition of $\mathrm{Alb} (X)$ above. 
Recall the following well-known result (see e.g.~\cite[Proposition 3.3.8]{Huybrechts}).

\begin{lemma}
	The Albanese map is holomorphic, and induces an isomorphism of vector spaces
	\begin{equation*}
		{\rm alb}^* \colon H^0 ( \mathrm{Alb} (X) , \Omega_{\mathrm{Alb} (X)}^1) \isom H^0 (X, \Omega_X^1).
	\end{equation*}
\end{lemma}

 As pointed out before, we always work with a fixed basis $\{ \zeta_j \}_{j=1}^d$ for $H^0(X, \Omega_X^1)$, which defines an isomorphism $\widetilde{\psi} \colon H^0(X, \Omega_X^1)^{\vee} \isom \mathbb{C}^d$ by means of the dual basis $\{ \zeta^{\vee}_j \}_{j=1}^d$. Suppose that we write $(z_1 , \dots , z_d )$ for the global holomorphic coordinates for $\mathbb{C}^d$ defined as $z_j = \widetilde{\psi} (\zeta^{\vee}_j)$. The choice of generators $[\gamma_1], [\gamma_2], \dots, [\gamma_{2d}]$ of the free part $H_1(X, \mathbb{Z})_{\rm free}$ defines a lattice $\Lambda:= \langle \lambda_1 , \dots , \lambda_{2d} \rangle$ in $\mathbb{C}^d$ via the relationship
 \begin{equation*}
 	\int_{\gamma_{\nu}} = \sum_{j=1}^{d} (\lambda_{\nu})_{j} \zeta_j^{\vee}
 \end{equation*}
 for each $\nu = 1 , \dots , 2d$, where $(\lambda_{\nu})_{j}$ is the $j$-th component of $\lambda_{\nu} \in \mathbb{C}^d$. All the arguments that follow depend on this specific basis and generators, and it is helpful to write down the above lemma in a more specific manner using these fixed data, which gives us an isomorphism $\psi \colon \mathrm{Alb} (X) \cong \mathbb{C}^d / \Lambda$ of complex Lie groups induced from $\widetilde{\psi}$. 
  
 \begin{lemma} \label{lmalbfxbsmp}
 	Fixing a basis $\{ \zeta_j \}_{j=1}^d$ for $H^0(X, \Omega_X^1)$ and generators $ \{ [\gamma_j] \}_{j=1}^{2d}$ of $H_1(X, \mathbb{Z})_{\rm free}$ determines an isomorphism $\psi \colon \mathrm{Alb} (X) \isom \mathbb{C}^d / \Lambda$ of complex Lie groups, where we write $(z_1 , \dots , z_d)$ for the global holomorphic coordinates for $\mathbb{C}^d$ given by the basis dual to $\{ \zeta_j \}_{j=1}^d$, such that the map
 	\begin{equation*}
 		{\rm alb}^* \circ \psi^* \colon H^0( \mathbb{C}^d / \Lambda , \Omega_{\mathbb{C}^d / \Lambda}^1) \isom H^0 (X, \Omega_X^1)
 	\end{equation*}
 	sends $d z_j$ (resp.~$d \overline{z_j}$) to $\zeta_j$ (resp.~$\overline{\zeta_j}$) for all $j=1 , \dots , d$.
 \end{lemma}
 
 \begin{proof}
 	By the functoriality of the Albanese map \cite[Proposition 3.3.8]{Huybrechts}, and recalling that the Albanese map is the identity for tori \cite[Example 3.3.9]{Huybrechts}, we find that the map $\psi \colon \mathrm{Alb} (X) \isom \mathbb{C}^d / \Lambda$ agrees with the one induced by the dual map
 \begin{equation*}
 	(\psi^*)^{\vee} \colon H^0 ( \mathrm{Alb} (X) , \Omega_{\mathrm{Alb} (X)}^1)^{\vee} \isom H^0 (\mathbb{C}^d / \Lambda , \Omega_{\mathbb{C}^d / \Lambda}^1)^{\vee}.
 \end{equation*}
 Combined with the natural isomorphism $H^0 (\mathbb{C}^d / \Lambda , \Omega_{\mathbb{C}^d / \Lambda}^1)^{\vee} = (\mathbb{C}^d)^{\vee \vee} = \mathbb{C}^d$ and the dual isomorphism $({\rm alb}^*)^{\vee} \colon H^0 (X, \Omega_X^1)^{\vee} \isom H^0 ( \mathrm{Alb} (X) , \Omega_{\mathrm{Alb} (X)}^1)^{\vee}$, we find that the map 
 \begin{equation*}
 	(\psi^*)^{\vee} \circ ({\rm alb}^*)^{\vee} \colon H^0(X, \Omega_X^1)^{\vee} \isom \mathbb{C}^d
 \end{equation*}
 must induce $\psi$, but this simply means $\widetilde{\psi} = ({\rm alb}^* \circ \psi^*)^{\vee}$. In particular, for all $j=1 , \dots , d$ we find $({\rm alb}^* \circ \psi^*)^{\vee} ( \zeta^{\vee}_j ) = z_j$, which in turn implies that the map
 \begin{equation*}
 	{\rm alb}^* \circ \psi^* \colon H^0( \mathbb{C}^d / \Lambda , \Omega_{\mathbb{C}^d / \Lambda}^1) = (\mathbb{C}^d)^{\vee} \isom H^0 (X, \Omega_X^1)^{\vee \vee} = H^0 (X, \Omega_X^1)
 \end{equation*}
 sends $d z_j$ to $\zeta_j$, by noting that the natural isomorphism $H^0( \mathbb{C}^d / \Lambda , \Omega_{\mathbb{C}^d / \Lambda}^1) = (\mathbb{C}^d)^{\vee}$ sends $\{ dz_j \}_{j=1}^d$ to the basis for $(\mathbb{C}^d)^{\vee}$ that is dual to $(z_1 , \dots , z_d )$. We also find
 \begin{equation*}
 	{\rm alb}^* \circ \psi^* ( d \overline{z_j} ) = \overline{\zeta_j}
 \end{equation*}
  for all $j=1 , \dots , d$, since the Albanese map and $\psi$ are both holomorphic.
 \end{proof}
 
 In what follows, we may identify $\mathrm{Alb} (X) = \mathbb{C}^d / \Lambda$ without explicitly writing down the isomorphism $\psi$, as we always work with the fixed basis $\{ \zeta_j \}_{j=1}^d$; in this notation the above lemma states
 \begin{equation} \label{eqalbfxbsmp}
 	{\rm alb}^* ( d z_j ) = \zeta_j , \quad {\rm alb}^* ( d \overline{z_j} ) = \overline{\zeta_j}
 \end{equation}
 for all $j=1 , \dots , d$.

We recall that $\mathrm{Pic}^0 (\mathrm{Alb}(X))$ is well-known to be isomorphic to $\mathrm{Pic}^0(X)$, as written e.g.~in \cite{GrHar}. For the reader's convenience we provide a detailed proof of this fact, and confirm that the isomorphism is given by the Albanese map. 

\begin{proposition}\label{prop:alb}
The morphism ${\rm alb}^*\colon {\rm Pic}^0(\mathrm{Alb} (X)) \to {\rm Pic}^0(X)$ in the category of abelian groups induced by the Albanese map is an isomorphism. 
\end{proposition}

\begin{proof}
For any compact K\"ahler manifold $Y$, we have an isomorphism of abelian groups
\begin{equation} \label{eqmdmmy}
	m_Y \colon \mathcal{P}(Y) \isom {\rm Hom}_{\rm group}(\pi_1(Y, *), {\rm U(1)})
\end{equation}
defined by the monodromy with respect to the Chern connection of a flat metric \cite[Proposition 2.2]{Koike}, where $\pi_1(Y, *)$ is the fundamental group of $Y$ with a fixed basepoint $*$. Note moreover that ${\rm alb}_*\colon H_1(X, \mathbb{Z})_{\rm free} \to H_1(\mathrm{Alb} (X), \mathbb{Z})$ is also an isomorphism, since for each loop $\gamma_\nu\colon [0, 1]\to X$ which represents the class $[\gamma_\nu] \in H_1 (X, \mathbb{Z})_{\rm free}$,
\[
 {\rm alb}\circ\gamma_\nu(t) = \left[ \zeta \mapsto \int_{\gamma_\nu|_t} \zeta \right]\in \mathrm{Alb} (X)\quad \left(\gamma_\nu|_t\colon [0, t]\ni s\mapsto \gamma(s)\in X\right)
\]
holds for each $t\in [0, 1]$, from which we immediately conclude that ${\rm alb}_*$ bijectively maps the generators of $H_1 (X, \mathbb{Z})_{\rm free}$ to the ones of $H_1(\mathrm{Alb} (X), \mathbb{Z})$, by the definition (\ref{eqdfalbv}). When we fix a basepoint $* \in X$ and ${\rm alb} (*) \in \mathrm{Alb} (X)$, the Hurewicz theorem implies that $H_1(\mathrm{Alb} (X), \mathbb{Z})$ agrees with the abelianisation of $\pi_1( \mathrm{Alb} (X), {\rm alb} (*))$, which gives us a natural isomorphism
\begin{equation*}
	{\rm Hom}_{\rm group}(\pi_1(\mathrm{Alb} (X), {\rm alb} (*)), {\rm U(1)}) \cong {\rm Hom}_{\rm group}(H_1(\mathrm{Alb} (X), \mathbb{Z}), {\rm U(1)}).
\end{equation*}
We apply the same argument to $\pi_1(X, *)$ to get another natural isomorphism
\begin{equation*}
	{\rm Hom}_{\rm group}(\pi_1(X, *), {\rm U(1)}) \cong {\rm Hom}_{\rm group}(H_1 (X, \mathbb{Z}), {\rm U(1)}). 
\end{equation*}

We thus get a sequence of isomorphisms
\begin{displaymath}
		\xymatrixcolsep{4pc}\xymatrixrowsep{4pc}\xymatrix{{\rm Hom}_{\rm group}(H_1(\mathrm{Alb} (X), \mathbb{Z}), {\rm U(1)}) \ar@{->}[r]^-{({\rm alb}_*)^{\vee}}_-{\cong} &  {\rm Hom}_{\rm group}(H_1 (X, \mathbb{Z})_{\rm free}, {\rm U(1)}) \ar@{_{(}-_>}[d]^-{m_X^{-1}}  \\
		 \mathcal{P}(\mathrm{Alb} (X)) \ar@{-->}[r] \ar@{->}[u]^-{m_{\mathrm{Alb} (X)}}_-{\cong} &  \mathcal{P} (X) 
} 
\end{displaymath}
where the dashed arrow above is defined by the composition of other isomorphisms, since \linebreak${\rm Hom}_{\rm group}(H_1 (X, \mathbb{Z})_{\rm free}, {\rm U(1)})$ can be naturally regarded as a subset of ${\rm Hom}_{\rm group}(H_1 (X, \mathbb{Z}), {\rm U(1)})$. 
As $\mathcal{P}(\mathrm{Alb} (X))={\rm Pic}^0(\mathrm{Alb} (X))$ holds by Lemma \ref{lem:kahler_pic0_p_relation} and the dashed arrow clearly agrees with ${\rm alb}^*$ by unravelling the definition of $m_X$ and $m_{\mathrm{Alb} (X)}$ (see e.g.~\cite[\S 2.1]{Koike}), 
it is sufficient to show that the image $I$ of the dashed arrow coincides with ${\rm Pic}^0(X)$, or equivalently, that
\begin{equation}\label{eq:isom_m_X_Pic0}
m_X^{-1}\colon {\rm Hom}_{\rm group}(H_1 (X, \mathbb{Z})_{\rm free}, {\rm U(1)}) \isom {\rm Pic}^0(X)
\end{equation}
is an isomorphism. 

As $\mathcal{P}(\mathrm{Alb} (X))={\rm Pic}^0(\mathrm{Alb} (X))$, it is clear that $I\subset {\rm Pic}^0(X)$ 
since the pull-back of a topologically trivial line bundle is topologically trivial. 
Therefore, in what follows, we assume for contradiction that $I\subsetneq {\rm Pic}^0(X)$. 
A standard argument using the exponential exact sequence implies that, for any torsion element $\alpha$ of $H^2 (X, \mathbb{Z})$, there exists an element $F \in {\rm Pic}(X)$ such that $\alpha=c_1^{\mathbb{Z}} (F)$, and also that such a line bundle $F$ is unique up to tensoring with an element of ${\rm Pic}^0(X)$. 
Note that such $F$ is an element of $\mathcal{P}(X)$ since $c_1(F)=0$. 
Therefore, combined with the isomorphism $H^2(X, \mathbb{Z})_{\rm tor}\cong H_1 (X, \mathbb{Z})_{\rm tor}$, which follows from the universal coefficient theorem, one obtains an injection $H_1 (X, \mathbb{Z})_{\rm tor} \to \mathcal{P}(X) / {\rm Pic}^0(X)$. 
From this injection and the injection $c_1^{\mathbb{Z}} \colon  \mathcal{P}(X) / {\rm Pic}^0(X) \to H_1 (X, \mathbb{Z})_{\rm tor} $, it turns out that the cardinalities of finite sets $H_1 (X, \mathbb{Z})_{\rm tor} $ and $ \mathcal{P}(X) / {\rm Pic}^0(X)$ coincide: 
\begin{equation}\label{eq_cardinality_h1_quot_of_p}
\#H_1 (X, \mathbb{Z})_{\rm tor} =\#\left( \mathcal{P}(X) / {\rm Pic}^0(X)\right). 
\end{equation}
Recall that we are assuming that $m_X^{-1}( {\rm Hom}_{\rm group}(H_1 (X, \mathbb{Z})_{\rm free}, {\rm U(1)})) \subsetneq {\rm Pic}^0(X)$. Thus it holds that 
\begin{equation*}
	m_X^{-1}( {\rm Hom}_{\rm group}(H_1 (X, \mathbb{Z}), {\rm U(1)}))/ m_X^{-1}( {\rm Hom}_{\rm group}(H_1 (X, \mathbb{Z})_{\rm free}, {\rm U(1)})) \supsetneq \mathcal{P}(X) / {\rm Pic}^0(X), 
\end{equation*}
which contradicts (\ref{eq_cardinality_h1_quot_of_p}) since the cardinality of the left hand side coincides with that of $H_1 (X, \mathbb{Z})_{\rm tor}$. 
\end{proof}

\subsection{Elementary results for tori}\label{section:elrtori}
We start with various elementary, but hopefully not entirely trivial, results in linear algebra. For the moment they are completely unmotivated, and are meant to serve as a general nonsense that is useful later.

\begin{lemma} \label{lmelivalmd}
Suppose that $\Lambda=\langle \lambda_1, \lambda_2, \dots, \lambda_{2d}\rangle$ is a lattice generated by a basis $\{ \lambda_k \}_{k=1}^{2d} \subset \mathbb{C}^d$ for the $\mathbb{R}$-vector space $\mathbb{C}^d$. Writing	$\lambda_1, \lambda_2, \dots, \lambda_{2d}$ as a a column vector consisting of $d$ complex numbers, the following real $2d \times 2d$ matrix 
\begin{equation*}
	A_{\Lambda} := \left(
\begin{array}{ccccc}
 & \mathrm{Re} \left( {^t\!\lambda_1} \right) & &  \mathrm{Im} \left( {^t\!\lambda_1} \right) & \\
 & \vdots & & \vdots & \\
 & \mathrm{Re} \left( {^t\!\lambda_{2d}} \right) & &  \mathrm{Im} \left( {^t\!\lambda_{2d}} \right) & \\
\end{array}
\right)
\end{equation*}
is invertible.
\end{lemma}

\begin{proof}
	Any zero row after elementary row operations on $A_{\Lambda}$ means non-trivial dependence relationship between $\lambda_1 , \dots , \lambda_{2d}$ over $\mathbb{R}$, which is a contradiction.
\end{proof}

\begin{remark}
	Note that the inverse matrix $A_{\Lambda}$ can be written in terms of the dual $\mathbb{R}$-basis $\{ \lambda^{\vee}_l \}_{l=1}^{2d}$ for $(\mathbb{C}^d)^{\vee}$, satisfying $\mathrm{Re} \left( {^t\!\lambda_j} \right) \cdot \mathrm{Re} \left( \lambda^{\vee}_l \right) + \mathrm{Im} \left( {^t\!\lambda_j} \right) \cdot \mathrm{Im} \left( \lambda^{\vee}_l \right) = \delta_{jl}$ for $j,l=1, \dots, 2d$, as
	\begin{equation} \label{eqivalmddl}
		A^{-1}_{\Lambda} =
		\left(
\begin{array}{ccccc}
 & \mathrm{Re} \left( \lambda^{\vee}_1 \right) & &  \mathrm{Re} \left( \lambda^{\vee}_{2d} \right) & \\
 & \mathrm{Im} \left( \lambda^{\vee}_1 \right) & &  \mathrm{Im} \left( \lambda^{\vee}_{2d} \right) & \\
\end{array}
\right).
	\end{equation}
\end{remark}

\begin{definition}
	Let $B \colon \mathbb{R}^{2d} = \mathbb{R}^d \times \mathbb{R}^d \to \mathbb{C}^d$ be a map defined by
\begin{equation*}
B\left(
\begin{array}{c}
a_1 \\ \vdots \\ a_d \\ 
b_1 \\ \vdots \\ b_d 
\end{array}
\right) = 
\frac{1}{2} \left(
\begin{array}{c}
a_1 + \ai b_1 \\ \vdots \\ a_d  + \ai b_d
\end{array}
\right),
\end{equation*}
which is clearly a linear isomorphism over $\mathbb{R}$. We define an $\mathbb{R}$-linear map
\begin{equation*}
	\bm{c} \colon \ai \mathbb{R}^{2d} \to \mathbb{C}^d
\end{equation*}
by 
\begin{equation*}
	\bm{c} ( \bm{s}):= \sqrt{-1} B \circ A^{-1}_{\Lambda} ( - \ai \bm{s})
\end{equation*}
for $\bm{s} \in \ai \mathbb{R}^{2d}$.
\end{definition}

Note that we can write $B$ as a $d \times 2d$ matrix $B = \frac{1}{2} \begin{pmatrix} I_d & \ai I_d \end{pmatrix}$, where $I_d$ is the $d \times d$ identity matrix and the usual multiplications of matrices is understood. 

The above definition looks artificial but it is useful for the computation later. We also have the following formula which may be conceptually more pleasing.

\begin{lemma} \label{lmexfcbdbs}
	Suppose that we identify $\mathbb{C}^d$ with its dual via the $\mathbb{R}$-bases $\{ \lambda_j \}_{j=1}^{2d}$ for $\mathbb{C}^d$ and $\{ \lambda^{\vee}_l \}_{l=1}^{2d}$ for $(\mathbb{C}^d)^{\vee}$. We then have 
	\begin{equation*}
		\bm{c} (\bm{s}) = \frac{1}{2} \sum_{l=1}^{2d} s_l \lambda^{\vee}_l \in (\mathbb{C}^d)^{\vee} \cong \mathbb{C}^d
	\end{equation*}
	where $\bm{s} = {^t(s_1 , \dots , s_{2d})} \in \ai \mathbb{R}^{2d}$.
\end{lemma}

\begin{proof}
	The claim follows immediately from the explicit formula (\ref{eqivalmddl}) using the dual $\mathbb{R}$-basis.
\end{proof}

\begin{lemma} \label{lmltcdltc}
	The map $\bm{c} \colon \ai \mathbb{R}^{2d} \to \mathbb{C}^d$ defined above is a linear isomorphism over $\mathbb{R}$. In particular,
	\begin{equation*}
		\Lambda_{\bm{c}} := \bm{c} (2 \pi \ai \mathbb{Z}^{2d}) \cong \mathbb{Z}^{2d}
	\end{equation*}
	defines a lattice in $\mathbb{C}^d$. Moreover, in the notation of Lemma \ref{lmexfcbdbs}, we have $\Lambda_{\bm{c}} \cong \ai \pi \Lambda^{\vee}$ where $\Lambda^{\vee}=\langle \lambda_1^{\vee}, \lambda_2^{\vee}, \dots, \lambda_{2d}^{\vee}\rangle$ is the dual lattice of $\Lambda$.
\end{lemma}

\begin{proof}
	Obvious from $B$ being a linear isomorphism over $\mathbb{R}$ and Lemma \ref{lmelivalmd}. Note that $\bm{c} (2 \pi \ai \mathbb{Z}^{2d})$ equals $\Lambda_{\bm{c}}$ by Lemma \ref{lmexfcbdbs}.
\end{proof}

The above lemmas indicate an important role played by the dual lattice $\Lambda^{\vee}$. It is indeed possible to write down various formulae below by using $\Lambda^{\vee}$ as opposed to using the slightly artificial definitions of $\bm{c}$ and $\Lambda_{\bm{c}}$, but we will not take this point of view since using the dual lattice involves constantly needing to identify $\mathbb{C}^d$ with its dual and we would like to treat $\bm{c} (\bm{s})$ as a $d$-tuple of complex numbers (rather than a dual vector). 

The following lemma is the key computation.

\begin{lemma} \label{lmkcalcs}
	For any $\bm{s} = {^t (s_1 , \dots , s_d)} \in \ai \mathbb{R}^d$, writing componentwise as $\bm{c} (\bm{s}) = (c_1 (\bm{s}), \dots , c_d (\bm{s}))$\linebreak $\in \mathbb{C}^d$, we have
\begin{equation*}
\left(
\begin{array}{ccccc}
 & {^t \lambda_1} & & {^t \overline{\lambda_1}} & \\
 & \vdots & & \vdots & \\
 & {^t \lambda_{2d}} & & {^t \overline{\lambda_{2d}}} & \\
\end{array}
\right)
\left(
\begin{array}{c}
-\overline{c_{1}(\bm{s})} \\ \vdots \\ -\overline{c_{d}(\bm{s})}\\ 
c_{1}(\bm{s})\\ \vdots \\ c_{d} (\bm{s})
\end{array}
\right)
= \bm{s}.
\end{equation*}
\end{lemma}

\begin{proof}
	We first define a $2d \times 2d$ complex matrix
\begin{equation*}
	E:= \begin{pmatrix} I_d &  I_d \\ \ai I_d & - \ai I_d \end{pmatrix}
\end{equation*}
in terms of block matrices, to observe
\begin{equation*}
\left(
\begin{array}{ccccc}
 & {^t \lambda_1} & & {^t \overline{\lambda_1}} & \\
 & \vdots & & \vdots & \\
 & {^t \lambda_{2d}} & & {^t \overline{\lambda_{2d}}} & \\
\end{array}
\right)
= A_{\Lambda} E .
\end{equation*}
Recalling that we can write $B$ as $\frac{1}{2} \begin{pmatrix} I_d & \ai I_d \end{pmatrix}$, we also find 
\begin{equation*}
\left(
\begin{array}{c}
- \overline{c_{1}(\bm{s})} \\ \vdots \\ - \overline{c_{d}(\bm{s})}\\ 
c_{1}(\bm{s})\\ \vdots \\ c_{d} (\bm{s})
\end{array}
\right)
= \frac{\ai }{2} \begin{pmatrix} I_d & - \ai  I_d \\ I_d & \ai I_d \end{pmatrix} A^{-1}_{\Lambda} ( - \ai \bm{s}),
\end{equation*}
to get
\begin{equation*}
	\frac{\ai}{2} A_{\Lambda} E \begin{pmatrix} I_d & - \ai  I_d \\ I_d & \ai I_d \end{pmatrix} A^{-1}_{\Lambda} ( - \ai \bm{s}) = \bm{s},
\end{equation*}
as claimed.
\end{proof}

We now define a torus 
\begin{equation*}
A:=\mathbb{C}^d/\Lambda , 	
\end{equation*}
where $\Lambda=\langle \lambda_1, \lambda_2, \dots, \lambda_{2d}\rangle$ is a lattice generated by an $\mathbb{R}$-basis $\lambda_1, \lambda_2, \dots, \lambda_{2d}$ of $\mathbb{C}^d$.

\begin{definition}
	We write $\mathcal{R}_{\pi_1} (A)$ for the set of all ${\rm U}(1)$-representations of $\pi_1(A, *)$.
\end{definition}

\begin{remark}
Note that $\mathcal{R}_{\pi_1} (A)$ is naturally an abelian group.
We do not need to pass to the quotient of $\mathcal{R}_{\pi_1} (\mathrm{Alb} (X))$ by the equivalence of representations since ${\rm U}(1)$ is abelian (and hence any equivalence class is a singleton).
\end{remark}

\begin{lemma} \label{lmctrisoqt}
	The map 
	\begin{equation*}
		\bm{c} \colon \mathcal{R}_{\pi_1} (A) \to \mathbb{C}^d / \Lambda_{\bm{c}}
	\end{equation*}
	given (by abuse of notation) by
	\begin{equation*}
		\bm{c} (\rho) := \bm {c} \left( \begin{array}{c} \log \rho (\lambda_1) \pmod{2 \pi \ai \mathbb{Z}} \\ \vdots \\ \log \rho (\lambda_{2d}) \pmod{2 \pi \ai \mathbb{Z}} \end{array} \right) \pmod{\Lambda_{\bm{c}}}
	\end{equation*}
	is well-defined and is an isomorphism of abelian groups.
\end{lemma}

\begin{proof}
	Note first that for each $k = 1 , \dots , 2d$, $\log \rho (\lambda_k)$ is not well-defined because of the branch locus, but $\log \rho (\lambda_k) \pmod{2 \pi \ai \mathbb{Z}}$ is well-defined and represented by a purely imaginary number. Hence the map $\bm{c}$ is a well-defined homomorphism of abelian groups by the definition of $\Lambda_{\bm{c}}$.
	
	Recall that any ${\rm U}(1)$-representation $\rho \in \mathcal{R}_{\pi_1} (A)$ factors through the abelianisation of $\pi_1 (A, *)$; $\rho$ determines and is determined by a homomorphism of abelian groups $H_1 (A , \mathbb{Z}) = \Lambda \to {\rm U}(1)$; this is a consequence of the Hurewicz theorem, which was already used in the proof of Proposition \ref{prop:alb}. Thus, given a set of generators $\lambda_1, \lambda_2, \dots, \lambda_{2d}$ for the lattice $\Lambda$, we find from the observation above that $\rho \in \mathcal{R}_{\pi_1} (A)$ is uniquely determined by the pairing with $\lambda_1, \lambda_2, \dots, \lambda_{2d}$, showing that $\bm{c}$ is bijective.
\end{proof}

%
\begin{remark} \label{rmfdeqbrlg}
The proof above immediately shows that choosing the fundamental domain of $\mathbb{C}^d / \Lambda_{\bm{c}}$ is equivalent to choosing the branch of $\log \rho$.
\end{remark}
%

As pointed out in the proof of Proposition \ref{prop:alb}, the monodromy map 
\begin{equation*}
	m_A \colon \mathcal{P} (A) \isom \mathcal{R}_{\pi_1} (A)
\end{equation*}
is an isomorphism of abelian groups. We write down this isomorphism more explicitly in the following lemma for the torus $A$; in fact, the proof of \cite[Proposition 2.2]{Koike}, which is used in the proof of Proposition \ref{prop:alb}, relies essentially on the same argument for an arbitrary compact K\"ahler manifold $Y$ to show that the map $m_Y$ in (\ref{eqmdmmy}) is an isomorphism.

\begin{lemma} \label{lmtrmndex}
	For each $F \in \mathcal{P} (A)$ there exists a unique $\rho := m_A (F) \in \mathcal{R}_{\pi_1} (A)$ such that $F$ is isomorphic to the line bundle
	\begin{equation*}
		F_\rho := (\mathbb{C}^d \times \mathbb{C})/\sim_\rho
	\end{equation*}
	over $A = \mathbb{C}^d / \Lambda$, where we define the equivalence relation $\sim_\rho$ by
\begin{equation*}
	(z, \xi) \sim_\rho (z+\lambda,\ \rho(\lambda)\cdot \xi)\quad (\forall\ \lambda\in \Lambda).
\end{equation*}
$F_\rho$ is endowed with a natural flat metric $h_\rho$ defined by
\begin{equation*}
\left|[(z, \xi)]\right|_{h_\rho} := |\xi|.
\end{equation*}

Conversely, for any $\rho \in \mathcal{R}_{\pi_1} (A)$ the line bundle $F_{\rho}$ constructed as above is isomorphic to $m_A^{-1} (\rho) \in \mathcal{P} (A)$.
\end{lemma}

\begin{proof}
Easy by checking the monodromy of $F_{\rho}$; see \cite[\S 2.1]{Koike} for more details.
\end{proof}

The previous computation yields the following.
\begin{lemma} \label{lmscsgmr}
Suppose that we write $F_\rho := (\mathbb{C}^d\times \mathbb{C})/\sim_\rho$ as above, and fix the generators $\lambda_1, \lambda_2, \dots, \lambda_{2d}$ for the lattice $\Lambda$. Suppose also that we choose a fundamental domain $D$ of $\mathbb{C}^d / \Lambda_{\bm{c}}$ so that the map $\bm{c} \colon \mathcal{R}_{\pi_1} (A) \to \mathbb{C}^d / \Lambda_{\bm{c}}$ in Lemma \ref{lmctrisoqt} can be written componentwise as
\begin{equation*}
			\bm{c} ( \rho ) = (c_1 (\rho) , \dots , c_d (\rho) ) \in D \subset \mathbb{C}^d.
\end{equation*}
Then the map $\sigma_{\rho} \colon A \to F_{\rho}$ defined by
	\begin{equation*}
		\sigma_{\rho} ([z]_1) := \left[\left(z,\ \exp \left(-\sum_{j=1}^d \overline{c_j(\rho)} \cdot z_j + \sum_{j=1}^d c_j(\rho)\cdot \overline{z_j} \right)\right)\right]_2 ,
	\end{equation*}
	where $[ z ]_1$ denotes the equivalence class of $z = (z_1 , \dots , z_d) \in \mathbb{C}^d$ for the quotient $\mathbb{C}^d / \Lambda =A$ and $[ \cdot ]_2$ for the quotient $(\mathbb{C}^d\times \mathbb{C})/\sim_\rho$, is a smooth non-vanishing section of $F_{\rho}$ of unit length with respect to $h_{\rho}$.
\end{lemma}

The section $\sigma_{\rho}$ of course depends on the particular generators $\lambda_1, \lambda_2, \dots, \lambda_{2d}$ for $\Lambda$, but we do not make it explicit in the notation since the generators for the lattice is always fixed in the applications of this result presented later.

\begin{proof}
	Lemma \ref{lmkcalcs} implies that, for each $k=1 , \dots , 2d$ and each $z \in \mathbb{C}^d$, we have
	\begin{align*}
		&\sigma_{\rho} ([z + \lambda_k ]_1) \\
		&= \left[\left(z + \lambda_k ,\ \exp \left(-\left( \sum_{j=1}^d \overline{c_j(\rho)} \cdot z_j +  {^t \lambda_k} \cdot \overline{\bm{c} (\rho)} \right) + \left( \sum_{j=1}^d c_j(\rho)\cdot \overline{z_j} +  \overline{^t \lambda_k} \cdot \bm{c} (\rho) \right) \right)\right)\right]_2  \\
		&= \left[\left(z + \lambda_k ,\ \rho(\lambda_k ) \exp \left(-\sum_{j=1}^d \overline{c_j(\rho)} \cdot z_j   + \sum_{j=1}^d c_j(\rho)\cdot \overline{z_j}  \right)\right)\right]_2 \\
		&= \left[\left(z ,\  \exp \left(-\sum_{j=1}^d \overline{c_j(\rho)} \cdot z_j  + \sum_{j=1}^d c_j(\rho)\cdot \overline{z_j}  \right)\right)\right]_2 \\
		&=\sigma_{\rho} ([z ]_1)
	\end{align*}
	which proves that $\sigma_{\rho}$ is indeed a well-defined global smooth section of $F_{\rho}$. It has unit length with respect to $h_{\rho}$ since the argument of the exponential map is clearly purely imaginary, and hence non-vanishing.
\end{proof}

The final key computation is the following.

\begin{lemma}\label{lem:lambda_c}
	Recall the isomorphism $\bm{c} \colon \mathcal{R}_{\pi_1} (A) \isom \mathbb{C}^d / \Lambda_{\bm{c}}$ in Lemma \ref{lmctrisoqt}, and suppose that we take a fundamental domain $D$ of $\mathbb{C}^d / \Lambda_{\bm{c}}$ whose interior contains the origin of $\mathbb{C}^d$. Suppose also that we write $F \in \mathcal{P}(A)$ as $F_{\rho}$ for a unique $\rho \in \mathcal{R}_{\pi_1} (A)$ by Lemma \ref{lmtrmndex}. Then, for any smooth $(p,q)$-form $\eta$ on $A$ and the section $\sigma_{\rho}$ of $F_{\rho}$ constructed in Lemma \ref{lmscsgmr}, we have
		\begin{equation*}
			 \delbar_{F_{\rho}} (\sigma_{\rho} \eta) = \sigma_{\rho} \cdot \left( \delbar \eta + \sum_{j=1}^d c_j(\rho)\cdot d\overline{z_j} \wedge \eta \right), 
		\end{equation*}
		where $ \delbar_{F_{\rho}}$ stands for the $\delbar$-operator on $A$ acting naturally on sections of $F_{\rho}$ and we wrote componentwise as
		\begin{equation*}
			\bm{c} ( \rho ) = (c_1 (\rho) , \dots , c_d (\rho) ) \in D \subset \mathbb{C}^d.
		\end{equation*}
\end{lemma}

\begin{proof}
We observe
\begin{equation*}
	\delbar_{F_{\rho}} (\sigma_{\rho} \eta) = ( \delbar_{F_{\rho}} \sigma_{\rho}) \wedge \eta + \sigma_{\rho} \delbar \eta
\end{equation*}
and compute, by recalling Lemma \ref{lmscsgmr},
\begin{equation*}
	\delbar_{F_{\rho}} \sigma_{\rho} = \sigma_{\rho} \left( \sum_{j=1}^d c_j(\rho)\cdot d\overline{z_j}\right),
\end{equation*}
which makes sense as long as we fix the fundamental domain $D$ of $\mathbb{C}^d / \Lambda_{\bm{c}}$. 
\end{proof}

\subsection{$\overline{\partial}$-operators on flat line bundles and the perturbed $\overline{\partial}$-operators} \label{section:dboflbpdb}

We now work on a compact K\"ahler manifold $X$ and a flat holomorphic line bundle $F$ over $X$. In Lemma \ref{lmhlstfdbf}, we observed that for a flat holomorphic line bundle $F$ that is holomorphically non-trivial, we have a $C^{\infty}$-bundle isomorphism $f \colon F \to \mathbb{I}_X$ such that $\delbar_{\mathbb{I}_X} \neq (f^{-1})^* \circ \delbar_F \circ f^*$. An important point here is that we can trade the $\delbar$-operator on $F$ for an operator $(f^{-1})^* \circ \delbar_F \circ f^*$ on $\mathbb{I}_X$ which can be written as a perturbation of the usual $\delbar$-operator on $\mathbb{I}_X$; working with operators on varying domains is difficult, but varying operators on a fixed domain is easier.

The ultimate aim of this section is to get a concrete description of $(f^{-1})^* \circ \delbar_F \circ f^*$ as a perturbation of the $\delbar$-operator (Proposition \ref{ppsmisolc} and Definition \ref{dfprdboprho}). The key step is in writing down the $C^{\infty}$-isomorphism $f$ explicitly by means of a smooth global section $\sigma\colon X\to F$ of unit length with respect to the flat metric, so that $f$ can be written more explicitly as $f \colon F \ni u\mapsto u/\sigma\in \mathbb{I}_X$. Such an explicit description of the perturbed $\delbar$-operators allows us to get concrete estimates for $(f^{-1})^* \circ \delbar_F \circ f^*$, which is what we discuss in \S \ref{scgaopdop}.

The Albanese variety and the materials we presented above play a crucial role in finding such an explicit description. We start by applying the results in the previous section for $A = \cx^d / \Lambda$ to a specific torus $\mathrm{Alb} (X)$.

\begin{proposition} \label{lmpcgpiso}
	Suppose that we fix a basis for $H^0 (X , \Omega_X^1)$ and generators of $H_1 (X, \mathbb{Z})_{\rm free}$ so that we identify $\mathrm{Alb} (X) = \mathbb{C}^d / \Lambda$. Then we have a commutative diagram of natural isomorphisms of abelian groups as follows. 
\begin{displaymath}
		\xymatrixcolsep{3pc}\xymatrixrowsep{4pc}\xymatrix{\mathbb{C}^d / \Lambda_{\bm{c}} \ar@{<-}[r]^-{\bm{c}}_-{\cong} & \mathcal{R}_{\pi_1} (\mathbb{C}^d / \Lambda ) \ar@{=}[r] & \mathcal{R}_{\pi_1} (\mathrm{Alb} (X)) \ar@{->}[r]^-{({\rm alb}_*)^{\vee}}_-{\cong} &  {\rm Hom}_{\rm group}(H_1 (X, \mathbb{Z})_{\rm free}, {\rm U(1)})  \ar@{<-}[d]^-{m_X}_-{\cong} \\ & & 
		 {\rm Pic}^0(\mathrm{Alb} (X)) \ar@{>}[r]^-{\cong}_-{{\rm alb}^*} \ar@{->}[u]^-{m_{\mathrm{Alb} (X)}}_-{\cong} &  {\rm Pic}^0 (X)} 
\end{displaymath}
In particular, ${\rm Pic}^0(X)$ is isomorphic as an abelian group to $\mathbb{C}^d / \Lambda_{\bm{c}}$, and by abuse of notation we write
\begin{equation*}
	\bm{c} \colon {\rm Pic}^0(X) \isom \mathbb{C}^d / \Lambda_{\bm{c}}
\end{equation*}
for the isomorphism.
\end{proposition}

It is of course obvious that ${\rm Pic}^0(X)$ is isomorphic to a torus, but the specific lattice $\Lambda_{\bm{c}}$ and the isomorphism $\bm{c}$ defined in Lemma \ref{lmctrisoqt} play a very important role later. 

\begin{proof}[Proof of Proposition \ref{lmpcgpiso}]
	Obvious from \cite[Proposition 2.2]{Koike}, Lemma \ref{lmctrisoqt}, and the proof of Proposition \ref{prop:alb}. Note that it follows from Lemma \ref{lem:kahler_pic0_p_relation} that $\mathcal{P}({\rm Alb}(X))={\rm Pic}^0({\rm Alb}(X))$, since the singular homology groups of a torus has no torsion. 
\end{proof}

\begin{remark}\label{rmk:amb_for_d} 
There is an ambiguity for the choice of the Euclidean distance $\mathsf{d}_0$ on ${\rm Pic}^0(X)$ from which we constructed the distance $\mathsf{d}$ in Theorem \ref{thm:main_1}, \ref{thm:main_2}, and Corollary \ref{cor:main}, since it is defined by choosing the Euclidean coordinates of the universal covering of ${\rm Pic}^0(X)$. 
This ambiguity causes no problem for the statements as any two Euclidean distances on ${\rm Pic}^0(X)$ are equivalent to each other, since any two norms on a finite dimensional vector space are equivalent. 
In the proof, 
we identify ${\rm Pic}^0(X)$ with $\mathbb{C}^d/\Lambda_{\bm c}$ via the isomorphism as in Proposition \ref{lmpcgpiso} and use the Euclidean distance $\mathsf{d}_0$ which comes from the standard coordinates of the universal covering $\mathbb{C}^d$ of $\mathbb{C}^d/\Lambda_{\bm c}$. 
\end{remark}

The above proposition immediately implies the following.

\begin{lemma} \label{lmsmisolc}
	Suppose $F \in {\rm Pic}^0(X)$. Then there exists unique $\rho \in \mathcal{R}_{\pi_1} (\mathrm{Alb} (X))$ such that $F = {\rm alb}^* F_{\rho} $, and moreover $h:= {\rm alb}^* h_{\rho}$ defines a flat metric on $F$ (which is unique up to scaling by Lemma \ref{lmuqflmfbf}). 
	
	Furthermore, the global non-vanishing section $\sigma_{F, \rho} := {\rm alb}^*\sigma_{\rho}$ of $F$ defines a $C^{\infty}$-bundle isomorphism $f_{\rho} \colon F \isom \mathbb{I}_X$ by
	\begin{equation*}
		f_{\rho}\colon F \ni u \mapsto \frac{u}{\sigma_{F, \rho}} \in \mathbb{I}_X,
	\end{equation*}
	which is an isometry with respect to the flat metric $h$ on $F$ and the standard metric on $\mathbb{I}_X$.
\end{lemma}

\begin{proof}
	The statement is just a summary of the arguments presented above; see in particular the proof of Proposition \ref{prop:alb} and \cite[Proposition 2.2]{Koike}, and note that ${\rm alb}^*\sigma_{\rho}$ clearly defines a flat metric on ${\rm alb}^* F_{\rho}$. For the claim that $f$ is an isometry, observe that the global section $\sigma_{X, \rho}$ of $F$ of unit length is mapped to $1 \in \mathbb{C}=H^0(X, \mathbb{I}_X)$.
\end{proof}

When $F$ is holomorphically non-trivial, we saw in Lemma \ref{lmhlstfdbf} that the $C^{\infty}$-isomorphism $f_{\rho} \colon F \isom \mathbb{I}_X$ defines an integrable partial connection $(f_{\rho}^{-1})^* \circ \delbar_F \circ f_{\rho}^*$ on $\mathbb{I}_X$. We can write it down more explicitly as follows.

\begin{proposition} \label{ppsmisolc}
	Suppose $F \in {\rm Pic}^0(X)$, which we may write as $F = {\rm alb}^* F_{\rho} $ for a unique $\rho \in \mathcal{R}_{\pi_1} (\mathrm{Alb} (X))$. Let $\zeta_1, \zeta_2, \dots, \zeta_d$ be a fixed basis for $H^0(X, \Omega_X^1)$, and we also fix generators of $H_1 (X, \mathbb{Z})_{\rm free}$ so that we identify $\mathrm{Alb} (X) = \mathbb{C}^d / \Lambda$, giving rise to a commutative diagram of isomorphisms 
\begin{displaymath}
		\xymatrixcolsep{8pc}\xymatrixrowsep{4pc}\xymatrix{ \mathcal{R}_{\pi_1} (\mathrm{Alb} (X)) \ar@{->}[r]^-{\rho \; \mapsto \; {\rm alb}^* F_{\rho} = F}_-{\cong} \ar@{->}[d]_-{\cong} &  {\rm Pic}^0 (X) \ar@{-->}[ld]^-{\bm{c}}_-{\cong} \\ \mathbb{C}^d / \Lambda_{\bm{c}}  & } 
\end{displaymath}
as in Proposition \ref{lmpcgpiso}. Suppose also that we choose a fundamental domain $D$ of $\mathbb{C}^d / \Lambda_{\bm{c}}$ whose interior contains the origin of $\mathbb{C}^d$ and write componentwise as 
		\begin{equation*}
			\bm{c} ( \rho ) = (c_1 (\rho) , \dots , c_d (\rho) ) \in D \subset \mathbb{C}^d.
		\end{equation*}
Then, for any $\eta \in A^{p,q} (X)$, the partial connection $(f_{\rho}^{-1})^* \circ \delbar_F \circ f_{\rho}^*$ on $\mathbb{I}_X$ acts on $\eta$ as
	\begin{equation} \label{eqprdfpdbop}
		(f_{\rho}^{-1})^* \circ \delbar_F \circ f_{\rho}^* \eta :=\delbar \eta + \sum_{j=1}^d c_j(\rho)\cdot \overline{\zeta_j} \wedge \eta ,
	\end{equation}
	where $c_1 (\rho ) = \cdots = c_d (\rho ) = 0$ if and only if $F$ is holomorphically trivial.
\end{proposition}

\begin{proof}
	For any $\eta \in A^{p,q}(X)$, $f^*_{\rho} \eta$ is an $F$-valued $(p,q)$-form $f^*_{\rho} \eta = \sigma_{F, \rho} \eta \in A^{p,q}(X,F)$. We then find
	\begin{align*}
		(f_{\rho}^{-1})^* \circ \delbar_F \circ f_{\rho}^* \eta &= (f_{\rho}^{-1})^* \circ \delbar_F ( \sigma_{F, \rho} \eta ) \\
		&=(f_{\rho}^{-1})^* ( f_{\rho}^* ( \delbar \eta ) ) + (f_{\rho}^{-1})^* (\delbar_F  \sigma_{F, \rho} ) \wedge \eta  \\
		&= \delbar \eta + {\rm alb}^* \left( \frac{1}{\sigma_{\rho}}\delbar_{F_{\rho}} \sigma_{\rho} \right) \wedge \eta,
	\end{align*}
	where we note $ {\rm alb}^* \delbar_{F_{\rho}} = \delbar_F$ since the Albanese map is holomorphic, and recall $\sigma_{F, \rho} = {\rm alb}^*\sigma_{\rho}$. Lemma \ref{lmalbfxbsmp} (see also (\ref{eqalbfxbsmp})) and the computation in the proof of Lemma \ref{lem:lambda_c} imply
	\begin{equation*}
		{\rm alb}^* \left( \frac{1}{\sigma_{\rho}}\delbar_{F_{\rho}} \sigma_{\rho} \right)  = {\rm alb}^* \left(  \sum_{j=1}^d c_j(\rho)\cdot d \overline{z_j} \right) = \sum_{j=1}^d c_j(\rho)\cdot \overline{\zeta_j},
	\end{equation*}
	which is the claimed result. We finally note that $F$ is holomorphically isomorphic to $\mathbb{I}_X$ if and only if $\rho$ is the trivial representation in $\mathcal{R}_{\pi_1} (\mathrm{Alb} (X))$, which occurs if and only if $\bm{c} (\rho) = 0$ in $\mathbb{C}^d$ mod $\Lambda_{\bm{c}}$. The choice of the fundamental domain $D$ implies that this happens if and only if $c_1 (\rho ) = \cdots = c_d (\rho ) = 0$.
\end{proof}


\begin{definition} \label{dfprdboprho}
	The partial connection (\ref{eqprdfpdbop}) in Proposition \ref{ppsmisolc}, defined with respect to the data therein, is called the \textbf{perturbed $\delbar$-operator} on $\mathbb{I}_X$ associated to $\rho \in \mathcal{R}_{\pi_1} (\mathrm{Alb} (X))$. In what follows we write 
	\begin{equation*}
		\delbar_{\rho} \colon A^{p,q} (X) \to A^{p,q+1}(X)
	\end{equation*}
	for $(f_{\rho}^{-1})^* \circ \delbar_F \circ f_{\rho}^*$, i.e.
	\begin{equation*}
		\delbar_{\rho} := \delbar  + \sum_{j=1}^d c_j(\rho)\cdot \overline{\zeta_j} \wedge .
	\end{equation*}
\end{definition}

\begin{remark}
	It is often convenient to write $\rho$ for the flat line bundle $F = {\rm alb}^* F_{\rho} \in {\rm Pic}^0 (X)$ associated to $\rho \in \mathcal{R}_{\pi_1} (\mathrm{Alb} (X))$. We may henceforth abuse the notation and say $\rho \in {\rm Pic}^0 (X)$ to mean the corresponding flat line bundle.
\end{remark}

Note that in the above $c_1 (\rho) , \dots , c_d (\rho )$ are just complex numbers for each $\rho \in {\rm Pic}^0 (X)$, with respect to the chosen fundamental domain, whose Euclidean norm is close to zero if and only if $\mathsf{d} ( \rho , \mathbb{I}_X)$ is small. This means that $\delbar_{\rho}$ is a benign perturbation of the $\delbar$-operator whose analytic properties are easy to establish, as we shall see in the next section, as long as $\rho$ is close to the trivial bundle in ${\rm Pic}^0(X)$ with respect to $\mathsf{d}_0$.


\subsection{Geometric analysis of the perturbed $\delbar$-operators} \label{scgaopdop}

Let $(X, \omega)$ be a compact K\"ahler manifold. 
Fixing a basis $\zeta_1, \zeta_2, \dots, \zeta_d$ for $H^0(X, \Omega_X^1)$ and generators $[\gamma_1], [\gamma_2], \dots, [\gamma_{2d}]$ of $H_1(X, \mathbb{Z})_{\rm free}$, here let us study the perturbed $\delbar$-operator $\delbar_\rho$ which we constructed in the previous section. 
We start with the following elementary lemma.

\begin{lemma}
	The adjoint of $\delbar_\rho \colon A^{p,q}(X) \to A^{p,q+1}(X)$ is given by $\delbar^*_{\rho} \colon A^{p,q}(X) \to A^{p,q-1}(X)$, which is defined as
\begin{equation*}
	\delbar^*_{\rho} := \delbar^* + \sum_{j=1}^d c_j (\rho) \ast \zeta_j \wedge \ast 
\end{equation*}
where $\ast$ is the Hodge star of the K\"ahler metric $\omega$, and $\delbar^* = - ( \ast \partial \ast)$ is the $\omega$-adjoint of $\delbar$. 
\end{lemma}

\begin{proof}
	For any $\alpha \in A^{p,q-1} (X)$ and $\beta \in A^{p,q} (X)$ we have
\begin{align*}
	\langle \overline{\zeta_j} \wedge \alpha , \beta \rangle_{\omega} &= \overline{\zeta_j} \wedge \alpha \wedge \overline{\ast \beta} \\
	&=(-1)^{p+q-1} \alpha \wedge \ast \ast^{-1} (\overline{\zeta_j \wedge \ast \beta}) \\
	&= (-1)^{p+q-1+n-p+n-q+1} \alpha \wedge \ast \overline{\ast (\zeta_j \wedge \ast \beta)} \\
	&= \langle \alpha , \ast (\zeta_j \wedge \ast \beta) \rangle_{\omega}, 
\end{align*}
from which the claim follows immediately.
\end{proof}

Throughout in what follows, we shall write $(-1)^{\bullet}$ for the sign factor that does not affect the argument. 

\begin{lemma} \label{lmlccp}
	Let $\eta_1, \eta_2$ be smooth $(0,1)$-forms and $\alpha_1 , \alpha_2$ be smooth $(p,q)$-forms. Then 
	\begin{equation*}
		|\langle \eta_1 \wedge \alpha_1 , \eta_2 \wedge \alpha_2 \rangle_{\omega} (x)| \le n \binom{n}{q} \cdot \Vert \eta_1 \Vert_{\omega} (x) \Vert \eta_2 \Vert_{\omega}(x) \Vert \alpha_1 \Vert_{\omega}(x) \Vert \alpha_2 \Vert_{\omega}(x) 
	\end{equation*}
	for any $x \in X$. Moreover, if $\alpha_1 , \alpha_2$ are smooth $(p,0)$-forms, we have
	\begin{equation*}
		\langle \eta_1 \wedge \alpha_1 , \eta_2 \wedge \alpha_2 \rangle_{\omega} (x)  = \langle \eta_1 , \eta_2  \rangle_{\omega} (x)  \cdot  \langle \alpha_1 , \alpha_2 \rangle_{\omega} (x)
	\end{equation*}
	for any $x \in X$.
\end{lemma}

\begin{proof}
	We fix a point $x \in X$ and work with holomorphic normal coordinates $(z_1 , \dots , z_n)$ around $x$. We then write
	\begin{equation*}
		\eta_l = \sum_{m=1}^n f^{(l)}_{m} d\overline{z_m} , \quad \alpha_l = \sum_{|I|=p, |J|=q} \alpha^{(l)}_{I,J} dz_{I} \wedge d\overline{z_{J}}
	\end{equation*}
	for $l=1,2$ and multi-indices $I = \{ i_1 , \dots , i_p \} \subset \{1 , \dots , n\}$ and $J = \{ j_1 , \dots , j_q \} \subset \{1 , \dots , n\}$. Then note
	\begin{equation*}
		\eta_l \wedge \alpha_l = (-1)^{p+q}\sum_{|I|=p, |J|=q, m \notin J}  f^{(l)}_{m} \alpha^{(l)}_{I,J}  dz_{I} \wedge d\overline{z_{J}} \wedge d\overline{z_m},
	\end{equation*}
	which implies
	\begin{equation*}
		|\langle \eta_1 \wedge \alpha_1 , \eta_2 \wedge \alpha_2 \rangle_{\omega} | = \left| \sum_{I, J, J', m, m'}  (-1)^{\bullet} f^{(1)}_{m} \alpha^{(1)}_{I,J} \overline{f^{(2)}_{m'} \alpha^{(2)}_{I,J'} }\right| \le  \sum_{I, J, J', m, m'}  |f^{(1)}_{m}| |f^{(2)}_{m'} | |\alpha^{(1)}_{I,J}| |\alpha^{(2)}_{I,J'}|
	\end{equation*}
	where the summation runs over all multi-indices $I$ with $|I|=p$, $J,J'$ with $|J|=|J'|=q$ and $m, m' \in \{1 , \dots , n \}$ such that
	\begin{equation*}
		\{ m \} \cup J = \{ m' \} \cup J' , \quad m \notin J, \quad \text{and} \quad m' \notin J',
	\end{equation*}
	noting that $\{ dz_{I} \wedge d\overline{z}_{J} \wedge d\overline{z}_m \}_{I,J,m}$ is an $\omega$-orthonormal basis for the set of $(p,q+1)$-forms at $x$. We then get, by Cauchy--Schwarz, 
	\begin{align*}
		&|\langle \eta_1 \wedge \alpha_1 , \eta_2 \wedge \alpha_2 \rangle_{\omega} | \\
		&\le \sum_{m,m'=1}^n |f^{(1)}_{m}| |f^{(2)}_{m'} | \sum_{|J|=|J'|=q} \sum_{|I|=p}   |\alpha^{(1)}_{I,J}| |\alpha^{(2)}_{I,J'}| \\
		&\le \left( \sum_{m=1}^n 1 \cdot |f^{(1)}_{m}| \right) \left( \sum_{m'=1}^n 1 \cdot |f^{(2)}_{m'} | \right) \left( \sum_{|J|=|J'|=q} \left( \sum_{|I|=p}|\alpha^{(1)}_{I,J}|^2 \right)^{1/2} \left(\sum_{|I|=p}  |\alpha^{(2)}_{I,J'}|^2 \right)^{1/2} \right) \\
		&\le n \left( \sum_{m=1}^n |f^{(1)}_{m}|^2 \right)^{1/2} \left( \sum_{m'=1}^n|f^{(2)}_{m'} |^2\right)^{1/2}  \left( \sum_{|J|=q} 1 \cdot \left( \sum_{|I|=p}|\alpha^{(1)}_{I,J}|^2 \right)^{1/2} \sum_{|J'|=q} 1 \cdot \left(\sum_{|I|=p}  |\alpha^{(2)}_{I,J'}|^2 \right)^{1/2} \right) \\
		&\le n \binom{n}{q} \cdot \left( \sum_{m=1}^n |f^{(1)}_{m}|^2 \right)^{1/2} \left( \sum_{m'=1}^n|f^{(2)}_{m'} |^2\right)^{1/2}    \left( \sum_{|I|=p, |J|=q}|\alpha^{(1)}_{I,J}|^2 \right)^{1/2} \left(\sum_{|I|=p, |J'|=q}  |\alpha^{(2)}_{I,J'}|^2 \right)^{1/2} \\
		&= n \binom{n}{q} \cdot \Vert \eta_1 \Vert_{\omega} \Vert \eta_2 \Vert_{\omega} \Vert \alpha_1 \Vert_{\omega} \Vert \alpha_1 \Vert_{\omega}
	\end{align*}
	as claimed.
	
	If $\alpha_1 , \alpha_2$ are smooth $(p,0)$-forms, we have
	\begin{equation*}
		\eta_l \wedge \alpha_l = (-1)^{p} \sum_{m=1}^n \sum_{|I|=p}  f^{(l)}_{m} \alpha^{(l)}_{I}  dz_{I} \wedge  d\overline{z}_m
	\end{equation*}
	with the notation as above, which implies
	\begin{equation*}
		\langle \eta_1 \wedge \alpha_1 , \eta_2 \wedge \alpha_2 \rangle_{\omega} = \sum_{m=1}^n \sum_{|I|=p}  f^{(1)}_{m} \overline{ f^{(2)}_{m}}  \cdot  \alpha^{(1)}_{I} \overline{\alpha^{(2)}_{I}}  = \langle \eta_1 , \eta_2  \rangle_{\omega}  \cdot  \langle \alpha_1 , \alpha_2 \rangle_{\omega}
	\end{equation*}
	as required.
\end{proof}

In what follows, we denote by $\Delta_{\rho} \colon A^{p,q} (X) \to A^{p,q} (X)$ the perturbed Laplacian defined by $\Delta_{\rho} : = \delbar^*_{\rho} \delbar_{\rho}  + \delbar_{\rho} \delbar^*_{\rho} $, 
and by $\mathbb{H}^{p,q} \subset A^{p,q} (X)$ the set of $\omega$-harmonic $(p,q)$-forms, which is the kernel of the ordinary Laplacian $\Delta : = \delbar^* \delbar + \delbar \delbar^*$. 

\begin{proposition} \label{ppelcpdedr}
Let $(\mathbb{H}^{p,q})^{\perp}$ be the orthogonal complement of $\mathbb{H}^{p,q}$ in $A^{p,q} (X)$.
Then, for any $\alpha \in (\mathbb{H}^{p,q})^{\perp}$ we have 
	\begin{equation*}
		(\alpha, \Delta_\rho\alpha)_{L^2} \ge \frac{1}{2} \left( \lambda_{p,q} - 4 n \binom{n}{q} C_X (\rho)^2 \right) \Vert \alpha \Vert_{L^2}^2 , 
	\end{equation*}
	where $\lambda_{p,q} >0$ is the smallest non-zero eigenvalue of $\Delta$ and
	\begin{equation*}
		C_X(\rho) := \sum_{j=1}^d |c_j (\rho)| \sup_{x \in X} \Vert \zeta_j \Vert_{\omega} (x).
	\end{equation*}
Note $\lambda_{p,q} - 4 n^2 C_X (\rho)^2 >0$ when $\rho$ is sufficiently close to the trivial representation.
Moreover, when $q=0$, we have a stronger inequality
	\begin{equation*}
		(\alpha, \Delta_\rho\alpha)_{L^2} \ge \frac{1}{2} \left( \lambda_{p,0} - 4  C_X (\rho)^2 \right) \Vert \alpha \Vert_{L^2}^2 .
	\end{equation*}
	for any $\alpha \in (\mathbb{H}^{p,0})^{\perp}$.
\end{proposition}

For the main results of this paper we only need the statement for $(p,0)$-forms, but we state and prove the general case for completeness. 

\begin{proof}
	First note
\begin{align*}
\|\delbar \alpha \|_{L^2}
&=\left\|\delbar_\rho \alpha  - \left(\sum_{j=1}^d c_j(\rho)\cdot \overline{\zeta_j}\right)\wedge\alpha\right\|_{L^2}\\
&\leq \|\delbar_\rho \alpha \|_{L^2} +  \sum_{j=1}^d |c_j (\rho)| \cdot \|  \overline{\zeta_j} \wedge \alpha \|_{L^2} \\
&= \|\delbar_\rho \alpha \|_{L^2} +  \sum_{j=1}^d |c_j (\rho)| \sqrt{\int_X \| \overline{\zeta_j} \wedge \alpha \|^2_{\omega} \frac{\omega^n}{n!}} \\
&\le \|\delbar_\rho \alpha \|_{L^2} + \left( r_{n,q} \sum_{j=1}^d |c_j (\rho)| \sup_{x \in X} \|  \zeta_j \|_{\omega}(x) \right) \|  \alpha \|_{L^2} \\
\end{align*}
by Lemma \ref{lmlccp}, where we set 
\begin{equation*}
	r_{n,q}:= \sqrt{n \binom{n}{q}}
\end{equation*}
to simplify the notation. Note that Lemma \ref{lmlccp} also implies that we have 
\begin{equation} \label{eqpzalp}
	\|\delbar \alpha \|_{L^2} \le \|\delbar_\rho \alpha \|_{L^2} + \left( \sum_{j=1}^d |c_j (\rho)| \sup_{x \in X} \|  \zeta_j \|_{\omega}(x) \right) \|  \alpha \|_{L^2}
\end{equation}
if $\alpha$ is a $(p,0)$-form. Similarly
\begin{align*}
	\|\delbar^* \alpha \|_{L^2} &\leq \|\delbar_\rho^*\alpha\|_{L^2}+ \sum_{j=1}^d |c_j (\rho)| \cdot \|  \ast (\zeta_j \wedge (\ast \alpha)) \|_{L^2} \\
	&= \|\delbar_\rho^*\alpha\|_{L^2}+ \sum_{j=1}^d |c_j (\rho)| \cdot \|  \overline{\zeta_j} \wedge \overline{(\ast \alpha)} \|_{L^2} \\
	&\le \|\delbar_\rho^*\alpha\|_{L^2} + \left(r_{n,q} \sum_{j=1}^d |c_j (\rho)| \sup_{x \in X} \|  \zeta_j \|_{\omega}(x) \right) \|  \alpha \|_{L^2}
\end{align*}
by Lemma \ref{lmlccp}, and by recalling that the complex conjugation and the Hodge star preserves the pointwise norm. Observe also that we have $\delbar^* \alpha = 0$ when $q=0$. We thus get
\begin{align*}
(\alpha, \Delta \alpha)_{L^2}
&=  \|\delbar \alpha \|_{L^2}^2+\|\delbar^*\alpha\|_{L^2}^2\\
&\leq \|\delbar_\rho \alpha \|_{L^2}^2+\|\delbar_\rho^*\alpha\|_{L^2}^2
+2r_{n,q} C_X(\rho)\cdot \|\alpha\|_{L^2}(\|\delbar_\rho \alpha\|_{L^2}+\|\delbar_\rho^* \alpha\|_{L^2})
+2r_{n,q}^2 C_X(\rho)^2\cdot \|\alpha\|_{L^2}^2 \\
&\leq 
2(\|\delbar_\rho\alpha\|_{L^2}^2+\|\delbar_\rho^*\alpha\|_{L^2}^2)+4r_{n,q}^2 C_X(\rho)^2 \cdot \|\alpha\|_{L^2}^2,
\end{align*}
where we note that we have
\begin{align*}
	&2r_{n,q} C_X(\rho)\cdot \|\alpha\|_{L^2} \cdot \|\delbar_\rho \alpha\|_{L^2}+ 2r_{n,q} C_X(\rho)\cdot \|\alpha\|_{L^2} \cdot \|\delbar_\rho^* \alpha\|_{L^2} \\
	&\le r_{n,q}^2 C_X(\rho)^2\cdot \|\alpha\|_{L^2}^2 +\|\delbar_\rho \alpha\|_{L^2}^2 + r_{n,q}^2 C_X(\rho)^2\cdot \|\alpha\|_{L^2}^2 + \|\delbar_\rho^* \alpha\|_{L^2}^2
\end{align*}
by the AM-GM inequality. Since $\alpha \in (\mathbb{H}^{p,q})^{\perp}$, we get
\begin{equation*}
	(\alpha, \Delta \alpha)_{L^2} \ge \lambda_{p,q} \cdot \Vert \alpha \Vert_{L^2}
\end{equation*}
which follows from the $L^2$-spectral decomposition theorem for $\Delta$, which is an elliptic self-adjoint linear operator acting on smooth sections of a vector bundle over a compact manifold. We thus get
\begin{align*}
\lambda_{p, q}\cdot \|\alpha\|_{L^2}^2&\leq (\alpha, \Delta \alpha)_{L^2}\\
&\leq 2(\|\delbar_\rho\alpha\|_{L^2}^2+\|\delbar_\rho^*\alpha\|_{L^2}^2)+4r_{n,q}^2C_X (\rho)^2 \cdot \|\alpha\|_{L^2}^2\\
&= 2(\alpha, \Delta_\rho\alpha)_{L^2}+4r_{n,q}^2C_X (\rho)^2 \cdot \|\alpha\|_{L^2}^2
\end{align*}
which immediately yields the claimed result. The improvement for $(p,0)$-forms is obvious from the estimates for $\|\delbar \alpha \|_{L^2}$ and $\|\delbar^* \alpha \|_{L^2}$ as in (\ref{eqpzalp}).
\end{proof}


\section{Proof of the main results} \label{section:proofmr}

In this section we prove Theorem \ref{thm:main_1}, \ref{thm:main_2} and Corollary \ref{cor:main}. 
Let $(X, \omega)$ be a compact K\"ahler manifold. We fix a basis $\zeta_1, \zeta_2, \dots, \zeta_d$ for $H^0(X, \Omega_X^1)$ and generators $[\gamma_1], [\gamma_2], \dots, [\gamma_{2d}]$ of $H_1(X, \mathbb{Z})_{\rm free}$, and use the notation in \S \ref{section:alb} and \S 3. We further choose $\{ \zeta_j \}_{j=1}^d$ to be an $L^2$-orthonormal basis with respect to the K\"ahler metric $\omega$.
Finally, we fix a fundamental domain $D$ of $\mathbb{C}^d/\Lambda_{\bm c}$ whose interior contains the origin, which is equivalent to choosing the branches of $\log\rho(\lambda_\nu)$'s so that $\log 1=0$, where $1\colon \Lambda\to \{1\}\subset {\rm U}(1)$ is the trivial representation. 

\subsection{Outline of proof}\label{section:proofmr_first_sub}
It follows from Proposition \ref{ppsmisolc} that it is sufficient to show the existence of a constant $K$ such that the inequality 
$\|\widehat{u}\|_{L^2}\leq K|1-\rho|^{-1}\cdot \|\delbar \widehat{u}\|_{L^2}$ holds for any ${\rm U}(1)$-representation $\rho$ of $\Lambda$ and any $\widehat{u}\in A^{p, 0}(X, {\rm alb}^*F_\rho)$, where we are letting 
\[
|1-\rho| := \sqrt{\sum_{j=1}^n|c_j(\rho)|^2}. 
\]
We first show this assertion only when $\rho$ is sufficiently close to the trivial representation $1$. 
Note that, as $c_j(1)=0$ holds by our choice of the branches of $\log\rho(\lambda_\nu)$'s, the distance on a neighbourhood of $\mathbb{I}_X$ in ${\rm Pic}^0(X)$ induced by the definition of $|1-\rho|$ above is equivalent to the Euclidean distance of ${\rm Pic}^0(X)$ (recall Lemma \ref{lmctrisoqt}, Remark \ref{rmfdeqbrlg}, and Remark \ref{rmk:amb_for_d}). 
Observe that $u:=\widehat{u}/\sigma_{X, \rho}\in A^{p, 0} (X)$ satisfies $\|\widehat{u}\|_{L^2, h}=\|u\|_{L^2}$ (with respect to the flat metric $h$ on the left hand side and the standard metric on $\mathbb{I}_X$ on the right hand side) and that we have
\[
\|\delbar \widehat{u}\|_{L^2,h}= \frac{\left\| \delbar \widehat{u} \right\|_{L^2, h}}{|\sigma_{X,\rho} |_{h}} 
=\|\delbar_\rho u\|_{L^2}\ \left(=\sqrt{(u, \Delta_\rho u)_{L^2}}\right), 
\]
since $|\sigma_{X,\rho} |_{h}$ is constantly equal to 1 over $X$; this is essentially the same as saying that the $C^{\infty}$-bundle isomorphism $f \colon F \to \mathbb{I}_X$ in Lemma \ref{lmsmisolc} is an isometry. Thus we get the desired result by showing that there exists a constant $K > 0$ such that the inequality 
\begin{equation}\label{eq:main_ineq}
\|u\|_{L^2}\leq \frac{K}{|1-\rho|} \|\delbar_\rho u\|_{L^2}\quad (u\in A^{p, 0} (X))
\end{equation}
holds for any ${\rm U}(1)$-representation $\rho$ of $\Lambda$ which is sufficiently close to the trivial representation $1$. 
We will show this assertion under either of the assumptions $(i)$ $p=0$ or $(ii)$ any holomorphic $p$-forms on $X$ are parallel. 
Note that it is well-known that holomorphic $p$-forms on $X$ are parallel by the Bochner vanishing if $\omega$ is Ricci-flat. Indeed, we have the Bochner--Weitzenb\"ock formula
\begin{equation*}
	2 \Delta u = \nabla^* \nabla u + \widetilde{R}(u)
\end{equation*}
where $\nabla^* \nabla$ is the rough Laplacian and $\widetilde{R}$ is an operator defined by contracting the curvature tensor of the K\"ahler metric $\omega$ and $u$ in an appropriate manner. It is well-known that $\widetilde{R}$ acting on $(p,0)$-forms depends only on the Ricci curvature of $\omega$ and that $\widetilde{R} (u) = 0$ if $u \in A^{p,0} (X)$ and $\mathrm{Ric} (\omega) = 0$. We thus get $2 \Delta u = \nabla^* \nabla u$ if $u \in A^{p,0} (X)$ and $\mathrm{Ric} (\omega) = 0$, which in turn yields
\begin{equation*}
	2 \| \delbar u \|^2 = \| \nabla u \|^2
\end{equation*}
showing that holomorphic $(p,0)$-forms are parallel. The reader is referred to \cite[Proposition 6.2.4]{Joyce}, which the above explanation closely followed, for more details. 

Take $u\in A^{p, 0} (X)$, and consider the orthogonal decomposition 
\[
u = \alpha + \beta,\quad \alpha\in (\mathbb{H}^{p, 0})^\perp,\ \beta\in \mathbb{H}^{p, 0}. 
\]
In order to show the inequality (\ref{eq:main_ineq}), we evaluate
\begin{equation*}
	\|\delbar_\rho u\|_{L^2}^2=\|\delbar_\rho \alpha + \delbar_\rho \beta\|_{L^2}^2 = \|\delbar_\rho \alpha\|_{L^2}^2 + 2{\rm Re}\,(\delbar_\rho \alpha, \delbar_\rho \beta)_{L^2} + \|\delbar_\rho \beta\|_{L^2}^2
\end{equation*}
in what we present below. We estimate $(\alpha, \Delta_\rho\alpha)_{L^2}$ by using Proposition \ref{ppelcpdedr}, and $(\beta, \Delta_\rho\beta)_{L^2}$ by using Lemma \ref{lmlccp}. The key result is the computation for $(\delbar_\rho \alpha, \delbar_\rho \beta)_{L^2}$, which we prove in Propositions \ref{prop:hashimoto2} and \ref{prop_hashimoto1}. The details are as follows.

\subsection{Proof of Theorem \ref{thm:main_1}}

Let $\alpha$ be a non-constant smooth function on $X$ with average zero and $\beta \in \cx$ be a constant, so that $\alpha \in  (\mathbb{H}^{0, 0})^\perp$ and $\beta \in  \mathbb{H}^{0, 0}$ in the notation above.

First note that
\begin{equation*}
	\delbar_\rho\beta = \beta \sum_{j=1}^d c_j(\rho) \overline{\zeta_j}
\end{equation*}
since $\beta$ is constant. Hence 
\begin{align}
	\|\delbar_\rho\beta\|^2_{L^2} &= |\beta|^2 \sum_{j,l=1}^d c_l (\rho) \overline{c_j (\rho)} \int_X \langle  \overline{\zeta_l} ,  \overline{\zeta_j} \rangle_{\omega}  \frac{\omega^n }{n!} \label{eqilbbeta} \\
	&= |\beta|^2 \sum_{j,l=1}^d c_l (\rho) \overline{c_j (\rho)} \int_X \langle \zeta_j , \zeta_l \rangle_{\omega}  \frac{\omega^n }{n!} \notag \\
	&= \frac{1}{{\rm Vol}_\omega(X)} |1 - \rho |^2 \cdot \| \beta \|^2_{L^2}, \notag
\end{align}
by recalling that $\zeta_1, \zeta_2, \dots, \zeta_d$ is an $L^2$-orthonormal basis for $H^0(X, \Omega_X^1)$.

Recall also from Proposition \ref{ppelcpdedr} that 
\[
 \| \delbar_\rho  \alpha\|_{L^2}^2 = (\alpha, \Delta_\rho\alpha)_{L^2} \ge \frac{1}{2} \left( \lambda_{0,0} - 4  C_X (\rho)^2 \right) \Vert \alpha \Vert_{L^2}^2
\]
holds for $C_X(\rho) = \sum_{j=1}^d |c_j (\rho)| \sup_{x \in X} \Vert \zeta_j \Vert_{\omega} (x)$, where $\lambda_{0, 0} >0$ is the smallest non-zero eigenvalue of $\Delta$ acting on functions. Noting $\sum_{j=1}^d |c_j (\rho)| \le \sqrt{d} |1 - \rho |$, we get
\begin{equation} \label{eqilbalpha}
	 \| \delbar_\rho \alpha\|_{L^2}^2 \ge \frac{1}{2} \left( \lambda_{0,0} - C_1  |1 - \rho |^2 \right) \Vert \alpha \Vert_{L^2}^2
\end{equation}
where we set 
\begin{equation*}
C_1= C_1(X, \omega, \{ \zeta_j \}_{j=1}^d):= 4 d \max_{l=1, \dots, d} \sup_{x \in X} \Vert \zeta_l \Vert^2_{\omega} (x) >0.
\end{equation*}

The following computation for the cross term is of crucial importance for us.

\begin{proposition} \label{prop:hashimoto2}
	Suppose that $\alpha$ is a non-constant smooth function on $X$ with average zero. Then we have
	\begin{equation*}
		(\delbar_{\rho} \alpha , \delbar_{\rho} \beta )_{L^2} = \overline{\beta} \sum_{j,l=1}^d c_l (\rho) \overline{c_j (\rho)} \int_X \langle  \zeta_j ,  \zeta_l \rangle_{\omega} \alpha  \frac{\omega^n }{n!} 
	\end{equation*}
	for any constant $\beta \in \mathbb{C}$.
\end{proposition}

\begin{proof}
	Since a constant function $\beta$ clearly satisfies $\delbar \beta = 0$, we have
	\begin{align*}
		(\delbar_{\rho} \alpha , \delbar_{\rho} \beta )_{L^2} &= \left( \delbar_{\rho} \alpha ,  \sum_{j=1}^d c_j (\rho) \overline{\zeta_j}  \wedge \beta \right)_{L^2} \\
		&= \left( \alpha , \delbar^* \left( \sum_{j=1}^d c_j (\rho) \overline{\zeta_j}  \wedge \beta \right) \right)_{L^2} + \sum_{j,l=1}^d c_l (\rho) \overline{c_j (\rho)} \left( \overline{\zeta_l}  \wedge \alpha , \overline{\zeta_j}  \wedge \beta \right)_{L^2}
	\end{align*}
	as before. Note further that we have
	\begin{equation*}
		\delbar^* \left( \sum_{j=1}^d c_j (\rho) \overline{\zeta_j}  \wedge \beta \right) = \beta  \sum_{j=1}^d c_j (\rho)  \delbar^* \overline{\zeta_j}  = 0.
	\end{equation*}
	Indeed, for each $j=1 , \dots , d$, $\zeta_j$ is a holomorphic $(1,0)$-form and hence $\delbar$-harmonic, which in turn implies that $\zeta_j$ is $d$-harmonic to yield 
	\begin{equation*}
		0 = ( \overline{\zeta}_j , \overline{ \Delta_d \zeta_j } )_{L^2}  =  ( \overline{\zeta}_j ,  \Delta_d \overline{\zeta}_j )_{L^2} = 2 ( \overline{\zeta}_j, \Delta \overline{\zeta}_j )_{L^2} = 2 \| \delbar \overline{\zeta}_j \|^2_{L^2} + \ 2 \| \delbar^* \overline{\zeta}_j \|^2_{L^2}
	\end{equation*}
	for all $j=1 , \dots , d$, where $\Delta_d := dd^* + d^*d$ is the $d$-Laplacian which is a real operator and equals $2 \Delta$ by the K\"ahler identity.
\end{proof}

The above proposition, together with Cauchy--Schwarz, implies the following key estimate
\begin{align}
		|(\delbar_{\rho} \alpha , \delbar_{\rho} \beta )_{L^2}| &\le |\beta| \sum_{j,l=1}^d |c_l (\rho)| |c_j (\rho)| \left| \int_X \langle  \zeta_j ,  \zeta_l \rangle_{\omega} \alpha   \frac{\omega^n }{n!} \right| \label{eqilbcrt} \\
		&\le \left( \sum_{j=1}^d |c_j(\rho)| \right)^2 \left( \max_{l=1, \dots, d} \sup_{x \in X} \Vert  \zeta_l  \Vert_{\omega}^2 (x) \right) \Vert \alpha \Vert_{L^2} \Vert \beta \Vert_{L^2} \notag \\
		&\le C_2 |1 - \rho |^2 \cdot \Vert \alpha \Vert_{L^2} \Vert \beta \Vert_{L^2} , \notag
\end{align}
where we set 
\begin{equation*}
	C_2=C_2(X, \omega, \{ \zeta_j \}_{j=1}^d):=d \max_{l=1, \dots, d} \sup_{x \in X} \Vert  \zeta_l  \Vert_{\omega}^2 (x)  >0 .
\end{equation*}

Combining the estimates (\ref{eqilbbeta}), (\ref{eqilbalpha}), (\ref{eqilbcrt}) we get
\begin{align*}
\|\delbar_\rho u\|_{L^2}^2
&=\|\delbar_\rho \alpha + \delbar_\rho \beta\|_{L^2}^2\\
&=\|\delbar_\rho \alpha\|_{L^2}^2 + 2{\rm Re}\,(\delbar_\rho \alpha, \delbar_\rho \beta)_{L^2} + \|\delbar_\rho \beta\|_{L^2}^2\\
&\geq \frac{1}{2} \left( \lambda_{0,0} - C_1  |1 - \rho |^2 \right) \|\alpha\|_{L^2}^2 + \frac{1}{{\rm Vol}_\omega(X)} |1-\rho|^2\cdot\|\beta\|_{L^2}^2 -  2\left|(\delbar_\rho \alpha, \delbar_\rho \beta)_{L^2} \right| \\
&\geq \frac{1}{2} \left( \lambda_{0,0} - C_1  |1 - \rho |^2 \right) \|\alpha\|_{L^2}^2 + \frac{1}{{\rm Vol}_\omega(X)} |1-\rho|^2\cdot\|\beta\|_{L^2}^2 - 2 C_2 |1 - \rho |^2 \cdot \Vert \alpha \Vert_{L^2} \Vert \beta \Vert_{L^2}
\end{align*}
for any $u \in A^{0, 0} (X)$, with $u = \alpha + \beta$, $\alpha\in (\mathbb{H}^{0, 0})^\perp$, $\beta\in \mathbb{H}^{0, 0}$.

For any $\ve >0$ we have
\begin{align*}
	2 C_2 |1 - \rho |^2 \cdot  \Vert \alpha \Vert_{L^2} \Vert \beta \Vert_{L^2} &= 2 \left( \sqrt{\frac{{\rm Vol}_\omega(X)}{\ve}} C_2 |1 - \rho | \cdot  \Vert \alpha \Vert_{L^2}\right) \left( \sqrt{\frac{\ve}{{\rm Vol}_\omega(X)}} |1 - \rho | \cdot \Vert \beta \Vert_{L^2} \right) \\
	&\le \frac{C^2_2}{\ve } {\rm Vol}_\omega(X)  |1 - \rho |^2 \cdot  \Vert \alpha \Vert^2_{L^2} + \frac{\ve}{{\rm Vol}_\omega(X)} |1 - \rho |^2 \cdot \Vert \beta \Vert^2_{L^2} 
\end{align*}
by the AM-GM inequality. We thus get, for any $0 < \ve < 1$,
\begin{equation*}
\|\delbar_\rho u\|_{L^2}^2 \geq \frac{1}{2} \left( \lambda_{0,0} - \left( C_1 + \frac{2 C^2_2}{\ve } {\rm Vol}_\omega(X) \right) |1 - \rho |^2 \right) \|\alpha\|_{L^2}^2 + \frac{1 - \ve }{{\rm Vol}_\omega(X)} |1-\rho|^2\cdot\|\beta\|_{L^2}^2.
\end{equation*}
Given $0 < \ve < 1$, we take $\rho$ to be sufficiently close to the identity so that
\begin{equation*}
	\frac{1}{2} \left( \frac{\lambda_{0,0}}{|1-\rho|^2} - \left(C_1+ \frac{2 C^2_2}{\ve }  {\rm Vol}_\omega(X) \right) \right) \ge \frac{1 - \ve }{{\rm Vol}_\omega(X)},
\end{equation*}
or equivalently
\begin{equation} \label{eqepsnbhid}
	|1 - \rho |^2 \le \frac{\ve \lambda_{0,0} {\rm Vol}_\omega(X)}{2 \ve  (1 - \ve ) + \ve {\rm Vol}_\omega(X)C_1 +2C^2_2{\rm Vol}_\omega(X)^2},
\end{equation}
which is possible since $\lambda_{0,0}, C_1 , C_2 , {\rm Vol}_\omega(X) >0$ are constants that depend only on $X$, $\omega$, and the $L^2$-orthonormal basis $\zeta_1, \zeta_2, \dots, \zeta_d$ for $H^0(X, \Omega_X^1)$. 

Thus, for any $0 < \ve < 1$, we define an open neighbourhood 
\begin{equation*}
	B_{\ve} := \left\{ \rho \in {\rm Pic}^0(X) \; \left| \; |1 - \rho |^2 < \frac{\ve \lambda_{0,0} {\rm Vol}_\omega(X)}{2 \ve  (1 - \ve ) + \ve {\rm Vol}_\omega(X)C_1 +2C^2_2{\rm Vol}_\omega(X)^2} \right\} \right.
\end{equation*}
around the identity in ${\rm Pic}^0(X)$ and take $\ve$ to be small enough so that $B_{\ve}$ is contained in the interior of the fundamental domain $D$. We then find that
\begin{align*}
\|\delbar_\rho u\|_{L^2}^2 &\geq |1-\rho|^2  \left( \frac{1}{2} \left( \frac{\lambda_{0,0}}{|1-\rho|^2} - \left( C_1 + \frac{2 C^2_2}{\ve }  {\rm Vol}_\omega(X) \right) \right)  \|\alpha\|_{L^2}^2 + \frac{1 - \ve }{{\rm Vol}_\omega(X)} \cdot\|\beta\|_{L^2}^2 \right) \\
&\geq \frac{1 - \ve }{{\rm Vol}_\omega(X)} |1-\rho|^2  \left(  \|\alpha\|_{L^2}^2 + \|\beta\|_{L^2}^2 \right) \\
&= \frac{1 - \ve }{{\rm Vol}_\omega(X)} |1-\rho|^2 \cdot \| u \|_{L^2}^2
\end{align*}
holds for all $\rho \in B_{\ve} \setminus \{ 1 \}$ and all $u \in A^{0, 0} (X)$, as required. 

For each $\rho \in \overline{D} \setminus \bm{c} (B_{\ve})$, let us consider 
\[
K_\rho = \sup\left\{\left.\frac{\|u\|_{L^2}}{\|\delbar_\rho u\|_{L^2}}\right|u\in ({\rm Ker}\,\delbar_\rho)^\perp\setminus\{0\}\right\}, 
\]
which satisfies $K_\rho<\infty$; see \cite[Theorem 8.37]{GT}. Here we regard $\rho$ as an element of the closure $\overline{D}$ of $D\subset \mathbb{C}^d$ via $\bm{c}$ so that $\delbar_{\rho}$ is well-defined; recall $\delbar_{\rho}$ depends only on $\bm{c} (\rho) \in D$, rather than $\rho\in{\rm Pic}^0(X)$ itself. There may be distinct elements in $\overline{D} \setminus \bm{c} (B_\ve )$ that are represented by the same $\rho\in{\rm Pic}^0(X)$, due essentially to the branch of $\log \rho$ as in Remark \ref{rmfdeqbrlg}, but this does not matter as long as we are concerned with the boundedness of $K_\rho$. 
Note that 
\begin{equation}\label{eq:expression_of_K_rho}
K_\rho = \sup\left\{\left.\frac{\|u\|_{L^2}}{\|\delbar_\rho u\|_{L^2}}\right|u\in A^{0, 0} (X) \setminus\{0\}\right\}
\end{equation}
holds by Lemma \ref{lem:flat_H0=0}. 

As the required estimate has already been shown on $B_\ve$ and 
$F\mapsto \mathsf{d}(\mathbb{I}_X, F)$ is bounded from below on ${\rm Pic}^0(X)\setminus B_\ve$, it is sufficient for proving Theorem \ref{thm:main_1} for the general case $F\in {\rm Pic}^0(X)\setminus\{\mathbb{I}_X\}$ to show the boundedness of $\{K_\rho \mid \rho\in \overline{D} \setminus \bm{c} (B_\ve )\}$ 
from above, which follows from Proposition \ref{prop:usc} below and the compactness of $\overline{D} \setminus \bm{c} (B_\ve ) $. 

\begin{proposition}\label{prop:usc}
The function $\rho \mapsto K_{\rho}$ 
is upper semi-continuous on  $\overline{D} \setminus \bm{c} (B_\ve )$. In particular, $\{K_\rho  \mid \bm{c} (\rho) \in \overline{D} \setminus \bm{c} (B_\ve ) \}$ is bounded from above. 
\end{proposition}

\begin{proof}
For a non-trivial ${\rm U}(1)$-representation $\rho$, we show the inequality 
\[
\limsup_{\rho'\to \rho}K_{\rho'}\leq K_{\rho}.
\]
Suppose for contradiction that there exits a positive number $\ve$ such that 
\[
\limsup_{\rho'\to \rho}K_{\rho'}\geq K_{\rho}+\ve
\]
holds. Then we can take a sequence $\{\rho_\nu\}$ of ${\rm U}(1)$-representations such that $\rho_\nu\to \rho$ with respect to $\mathsf{d}_0$ as $\nu\to \infty$ and that 
\[
K_{\rho_\nu}\geq K_{\rho}+\frac{\ve}{2}
\]
holds for any $\nu$. Thus, for any $\nu$, there exists an element $u_\nu\in A^{0, 0} (X)$ such that 
\[
\frac{\|u_\nu\|_{L^2}}{\|\delbar_{\rho_\nu}u_\nu\|_{L^2}}\geq K_\rho+\frac{\ve}{3} 
\]
holds. Without loss of generality, we may assume $\|u_\nu\|_{L^2}=1$. 

 We may assume that $\eta_\nu := \sum_{j=1}^d\left(c_j(\rho_\nu)-c_j(\rho)\right)\overline{\zeta_j}$ satisfies $\sup_{x\in X}\|\eta_\nu\|_\omega(x)\to 0$ as $\nu\to 0$.
Then, as 
$\delbar_{\rho_\nu} u_\nu = \delbar_\rho u_\nu + \eta_\nu\wedge u_\nu$, it follows from Lemma \ref{lmlccp} that 
\begin{align*}
1=\|u_\nu\|_{L^2}&\geq \left(K_\rho+\frac{\ve}{3}\right)\cdot \|\delbar_{\rho_\nu}u_\nu\|_{L^2}\\
&= \left(K_\rho+\frac{\ve}{3}\right)\cdot \|\delbar_\rho u_\nu + \eta_\nu\wedge u_\nu\|_{L^2}\\
&\geq \left(K_\rho+\frac{\ve}{3}\right)\cdot \left(\|\delbar_\rho u_\nu\|_{L^2} - \|\eta_\nu\wedge u_\nu\|_{L^2}\right)\\
&\geq \left(K_\rho+\frac{\ve}{3}\right)\cdot \left(\|\delbar_\rho u_\nu\|_{L^2} - \sup_{x\in X}\|\eta_\nu\|_\omega(x)\cdot \|u_\nu\|_{L^2}\right)\\
&= \left(K_\rho+\frac{\ve}{3}\right)\cdot \left(\|\delbar_\rho u_\nu\|_{L^2} -\sup_{x\in X}\|\eta_\nu\|_\omega(x)\right)
\end{align*}
holds. 
As the equation (\ref{eq:expression_of_K_rho}) implies 
\[
\|\delbar_\rho u_\nu\|_{L^2}\geq \frac{\|u_\nu\|_{L^2}}{K_\rho}=\frac{1}{K_\rho}, 
\]
one has that
\[
1\geq \left(1+\frac{\ve}{3}\cdot \frac{1}{K_\rho}\right)\cdot \left(1 - K_\rho\cdot \sup_{x\in X}\|\eta_\nu\|_\omega(x)\right)
\]
holds for any $\nu$, which leads to contradiction since $\sup_{x\in X}\|\eta_\nu\|_\omega(x)\to 0$ as $\nu\to 0$.

The upper semi-continuity established above proves the boundedness result since $\overline{D} \setminus \bm{c} (B_\ve )$ is compact in $\mathbb{C}^d$ with respect to $\mathsf{d}$. 
\end{proof}

As pointed out above, the proof of Proposition \ref{prop:usc} completes the proof of the statements claimed in Theorem \ref{thm:main_1} when $F\in {\rm Pic}^0(X)\setminus\{\mathbb{I}_X\}$. 

\begin{remark}\label{rmk:final_rmk_thm_1_1_i}
For a flat line bundle $F$ on $X$ and a flat metric $h$, 
denote by $K(F)$ the minimum of all the positive constants $M$ such that the condition {\bf ($L^2$-estimate)$_M$} holds for $(X, F, \omega, h)$, i.e.~$(K(F))^2$ is the minimum positive eigenvalue of the Laplace operator $\delbar^*_F\delbar_F$ acting on $A^{0, 0}(X, F)$. 
As the solution of the $\delbar$-equation with minimum $L^2$-norm must be perpendicular to $\mathbb{H}^{p, q}$, it follows that $K_\rho=K({\rm alb}^*F_\rho)$ holds for each $\rho\in{\rm Pic}^0(X)$ (recall the argument that precedes (\ref{eq:main_ineq}) in \S \ref{section:proofmr_first_sub}), i.e.~$K_\rho$ does not depend on the choice of the branch of $\log \rho$ (although the operator $\delbar_\rho$ appears in the definition of $K_\rho$), which means that the function $\rho\mapsto K_\rho$ is well-defined on ${\rm Pic}^0(X)$. 
Proposition \ref{prop:usc} implies that this function is upper semi-continuous on ${\rm Pic}^0(X)\setminus \{1\}$. 
\end{remark}

Finally let us show the assertion of Theorem \ref{thm:main_1} when $F\in {\rm Pic}^t(X)$ for some $t\in G\setminus\{0\}$ (it is sufficient to show by fixing $t$ since $G$ is a finite group by Lemma \ref{lem:kahler_pic0_p_relation}). 
Again, as $F\mapsto \mathsf{d}(\mathbb{I}_X, F)$ is bounded from below on ${\rm Pic}^t(X)$, it is sufficient to show that 
\begin{equation}\label{eq:bddness_of_supK(F)_on_Pict}
\sup\{K(F)\mid F\in {\rm Pic}^t(X)\} < \infty
\end{equation}
holds, where $K(F)$ is as in Remark \ref{rmk:final_rmk_thm_1_1_i}, i.e. 
\[
K(F) := \sup\left\{\left.\frac{\|f\|_{h}}{\|\delbar_Ff\|_{h, \omega}}\right|f\in\left({\rm Ker}\,\delbar_F\right)^\perp\setminus\{0\}\right\}, 
\]
where $h$ is a flat metric on $F$ and $\|\bullet\|_h$ ($\|\bullet\|_{h, \omega}$) denotes the $L^2$ norm with respect to $h$ (and $\omega$). 
Note that 
\begin{equation}\label{eq:expression_of_K_rho2}
K(F) = \sup\left\{\left.\frac{\|f\|_{h}}{\|\delbar_Ff\|_{h, \omega}}\right|f\in A^{0, 0} (X, F)\setminus\{0\}\right\}
\end{equation}
holds by Lemma \ref{lem:flat_H0=0}. 
As ${\rm Pic}^t(X)$ is compact, 
(\ref{eq:bddness_of_supK(F)_on_Pict}) follows from Proposition \ref{prop:usc2} below.

\begin{proposition}\label{prop:usc2}
The function $F \mapsto K(F)$ 
is upper semi-continuous on ${\rm Pic}^t(X)$. In particular, (\ref{eq:bddness_of_supK(F)_on_Pict}) holds. 
\end{proposition}

\begin{proof}
For an element $F\in {\rm Pic}^t(X)$, we show the inequality 
\[
\limsup_{F'\to F}K(F')\leq K(F)
\]
by contradiction. Assume there exits a positive number $\ve$ such that 
\[
\limsup_{F'\to F}K(F')\geq K(F)+\ve
\]
holds. 
Then one can take a sequence $\{\rho_\nu\}$ of ${\rm U}(1)$-representations such that $\rho_\nu\to 1$ as $\nu\to \infty$ and that 
\[
K(F\otimes {\rm alb}^*F_{\rho_\nu})\geq K(F)+\frac{\ve}{2}
\]
holds for any $\nu$, since $F'\otimes F^{-1}\in {\rm Pic}^0(X)$ for any $F'\in {\rm Pic}^t(X)$. 
In what follows, we let $F_\nu:=F\otimes {\rm alb}^*F_{\rho_\nu}$. 

Here we note that, by considering $u:=f/\sigma_{{\rm alb}^*F_{\rho_\nu}, \rho_\nu}\in A^{0, 0}(X, F)$ for each $f\in A^{0, 0}(X, F_\nu)$, we can deduce from (\ref{eq:expression_of_K_rho2}) that 
\[
K(F_\nu) = \sup\left\{\left.\frac{\|u\|_{h}}{\|\delbar_Fu+\eta_\nu\wedge u\|_{h, \omega}}\right|u\in A^{0, 0} (X, F)\setminus\{0\}\right\}, 
\]
where $\eta_\nu := \sum_{j=1}^dc_j(\rho_\nu)\overline{\zeta_j}$, since 
\begin{align*}
\delbar_{F_\nu} (u \sigma_{{\rm alb}^*F_{\rho_\nu}})
&=(\delbar_Fu)\wedge \sigma_{{\rm alb}^*F_{\rho_\nu}}+u\wedge (\delbar_{{\rm alb}^*F_{\rho_\nu}}\sigma_{{\rm alb}^*F_{\rho_\nu}})\\
&=\left(\delbar_Fu + u\wedge \frac{\delbar_{{\rm alb}^*F_{\rho_\nu}}\sigma_{{\rm alb}^*F_{\rho_\nu}}}{\sigma_{{\rm alb}^*F_{\rho_\nu}}} \right)\cdot \sigma_{{\rm alb}^*F_{\rho_\nu}}\\
&=\left(\delbar_Fu + u\wedge \sum_{j=1}^d c_j(\rho_\nu)\cdot \overline{\zeta_j} \right)\otimes \sigma_{{\rm alb}^*F_{\rho_\nu}}
\end{align*}
holds by (\ref{eqprdfpdbop}). 
Note also that we may assume that $\sup_{x\in X}\|\eta_\nu\|_\omega(x)\to 0$ as $\nu\to 0$.

Therefore, the rest of the proof can be done in the same manner as that of Proposition \ref{prop:usc}, by simply replacing $\rho$ with $1$, $K_\rho$ with $K(F)$, $\delbar_\rho$ with $\delbar_F$, and $\delbar_{\rho_\nu}$ with $\delbar_F+\eta_\nu\wedge$. Note that the inequality 
\[
\|\eta_\nu\wedge u_\nu\|_{h, \omega}\leq \sup_{x\in X}\|\eta_\nu\|_\omega(x)\cdot \|u_\nu\|_{h}
\]
for $u_\nu\in A^{0, 0}(X, F)$, which is needed in the proof, can be shown by the same argument as in the proof of Lemma \ref{lmlccp}.
\end{proof}

\begin{remark}
One can show the boundedness (\ref{eq:bddness_of_supK(F)_on_Pict}) also by constructing suitable finite coverings $\pi\colon \widetilde{X}\to X$ such that $\pi^*({\rm Pic}^t(X))\subset {\rm Pic}^0(\widetilde{X})$ and applying Proposition \ref{prop:usc} to ${\rm Pic}^0(\widetilde{X})$. 
See Appendix for the details. 
\end{remark}

As is mentioned above, the proof of Proposition \ref{prop:usc2} completes the proof of all the statements claimed in Theorem \ref{thm:main_1}. 

\subsection{Proof of Theorem \ref{thm:main_2}}

First note that Lemma \ref{lmlccp} implies
\begin{align*}
\|\delbar_\rho\beta\|^2_\omega &= \left\| \sum_{j=1}^d c_j(\rho)\cdot \overline{\zeta_j} \wedge \beta\right\|^2_\omega \\
&= \sum_{j,l=1}^d c_j (\rho) \overline{c_l (\rho)} \langle  \overline{\zeta_j} \wedge \beta , \overline{\zeta_l} \wedge \beta \rangle_{\omega} \\
&=\sum_{j,l=1}^d c_j (\rho) \overline{c_l (\rho)}  \langle  \zeta_l  ,  \zeta_j \rangle_{\omega} \| \beta \|^2_\omega .
\end{align*}
Recalling that $\zeta_1, \zeta_2, \dots, \zeta_d$ are parallel differential forms with respect to $\omega$, we find that $\langle  \zeta_l  ,  \zeta_j \rangle_{\omega}$ is a constant over $X$, i.e.~for any $x \in X$ we have
\begin{equation*}
	\langle  \zeta_l  ,  \zeta_j \rangle_{\omega} (x ) = \frac{1}{{\rm Vol}_\omega(X)}( \zeta_l , \zeta_j )_{L^2} = \frac{1}{{\rm Vol}_\omega(X)} \delta_{lj}
\end{equation*}
where the right hand side is the Kronecker delta. We thus get
\begin{equation} \label{eqiilbbeta}
	\|\delbar_\rho\beta\|^2_{L^2} = \frac{1}{{\rm Vol}_\omega(X)} |1 - \rho |^2 \cdot \| \beta \|^2_{L^2}.
\end{equation}

Note again that Proposition \ref{ppelcpdedr} implies
\begin{equation*}
	\| \delbar_\rho \alpha\|_{L^2}^2 = (\alpha, \Delta_\rho\alpha)_{L^2} \ge \frac{1}{2} \left( \lambda_{p,0} - 4  C_X (\rho)^2 \right) \Vert \alpha \Vert_{L^2}^2
\end{equation*}
holds for $C_X(\rho) = \sum_{j=1}^d |c_j (\rho)| \sup_{x \in X} \Vert \zeta_j \Vert_{\omega} (x)$, where $\lambda_{0, 0} >0$ is the smallest non-zero eigenvalue of $\Delta$ acting on $(p,0)$-forms. Noting $\sum_{j=1}^d |c_j (\rho)| \le \sqrt{d} |1 - \rho |$, we get
\begin{equation} \label{eqiilbalpha}
	\| \delbar_\rho \alpha\|_{L^2}^2 \ge \frac{1}{2} \left( \lambda_{p,0} - C_3  |1 - \rho |^2 \right) \Vert \alpha \Vert_{L^2}^2
\end{equation}
where we set 
\begin{equation*}
C_3 =  C_3(X, \omega, \{ \zeta_j \}_{j=1}^d):= 4 d  \max_{l=1, \dots, n} \sup_{x \in X} \Vert \zeta_l \Vert^2_{\omega} (x) >0 .
\end{equation*}

As in Proposition \ref{prop:hashimoto2} that we proved in the previous section, evaluation of the cross term is of crucial importance. We have a stronger result now that we assume that all holomorphic $p$-forms are parallel.

\begin{proposition}\label{prop_hashimoto1}
	Suppose that all holomorphic differential forms on $X$ are parallel with respect to the K\"ahler metric $\omega$. 
Then we have
	\begin{equation*}
		(\delbar_{\rho} \alpha , \delbar_{\rho} \beta )_{L^2} = 0
	\end{equation*}
	for any $\beta \in H^0 (X , \Omega_X^{p})$ and any $\alpha \in A^{p,0} (X)$ that is $L^2$-orthogonal to holomorphic $p$-forms.
\end{proposition}

\begin{proof}
	First note that we have $\delbar \beta = \delbar^* \beta = 0$ since $\beta$ is $\delbar$-harmonic, which implies
	\begin{align*}
		(\delbar_{\rho} \alpha , \delbar_{\rho} \beta )_{L^2} &= \left( \delbar_{\rho} \alpha ,  \sum_{j=1}^d c_j (\rho) \overline{\zeta_j}  \wedge \beta \right)_{L^2} \\
		&= \left( \alpha , \delbar^* \left( \sum_{j=1}^d c_j (\rho) \overline{\zeta_j}  \wedge \beta \right) \right)_{L^2} + \sum_{j,l=1}^d c_l (\rho) \overline{c_j (\rho)} \left( \overline{\zeta_l}  \wedge \alpha , \overline{\zeta_j}  \wedge \beta \right)_{L^2} .
	\end{align*}
	
	At each point $x \in X$, we pick a normal holomorphic coordinate system $(z_1 , \dots , z_n)$ which defines an $\omega$-orthonormal basis $d z_1 , \dots , d z_n$ for the holomorphic cotangent space $T^*_{X,x}$ and write $\zeta_j$ and $\beta$ locally in terms of the chosen basis. A crucially important fact is that we can write
	\begin{equation*}
		\zeta_j = \sum_{m=1}^n f_{j,m} dz_m , \quad \beta = \sum_{1 \le i_1 , \dots , i_p \le n} \beta_{i_1, \dots , i_p} dz_{i_1} \wedge \cdots dz_{i_p}
	\end{equation*}
	for smooth $\mathbb{C}$-valued functions $f_{j,m}$ and $\beta_{i_1, \dots , i_p}$ defined locally around $x \in X$ such that
	\begin{equation} \label{eqprll}
		\frac{\partial f_{j,m}}{\partial z_k} (x) = \frac{\partial f_{j,m}}{\partial \overline{z}_k} (x) = \frac{\partial \beta_{i_1, \dots , i_p}}{\partial z_k} (x) = \frac{\partial \beta_{i_1, \dots , i_p}}{\partial \overline{z}_k} (x) = 0
	\end{equation}
	holds at $x \in X$ and for any $k=1 , \dots , n$, since $\zeta_j$, $\beta$ are holomorphic and hence parallel by the hypothesis; note that the covariant derivative agrees with the ordinary coordinate derivatives at $x$ since we took the normal holomorphic coordinates. Using the multi-index notation $I := \{ i_1 , \dots , i_p \}$, we thus have
	\begin{align*}
		\ast (\overline{\zeta_j} \wedge \beta ) &= \ast \sum_{m=1}^n \sum_I (-1)^p \overline{f_{j,m}} \beta_I  dz_I \wedge  d \overline{z_m} \\
		&=\sum_{m=1}^n \sum_I (-1)^{\bullet} \overline{f_{j,m}} \beta_{I}  dz_{\{1, \dots , n \} \setminus \{ m \}} \wedge  d \overline{z_{\{ 1 , \dots , n\} \setminus I}},
	\end{align*}
	hence
	\begin{align*}
		\left. \delbar^* \left( \sum_{j=1}^n c_j (\rho) \overline{\zeta_j}  \wedge \beta \right) \right|_x &= - \left. \sum_{j=1}^n c_j (\rho) \ast \partial \ast (\overline{\zeta_j} \wedge \beta ) \right|_x \\
		&= -  \sum_{j=1}^n c_j (\rho) \sum_{m=1}^n \sum_I (-1)^{\bullet} \left. \ast\partial (\overline{f}_{j,m} \beta_{I} )\right|_x \wedge dz_{\{1, \dots , n \} \setminus \{ m \}} \wedge  d \overline{z}_{\{ 1 , \dots , n\} \setminus I} \\
		&=0
	\end{align*}
	since
	\begin{equation*}
		\left. \partial (\overline{f}_{j,m} \beta_{I} )\right|_x = \left. \partial (\overline{f}_{j,m}  )\right|_x \beta_{I} (x) + \overline{f}_{j,m} (x) \left. \partial ( \beta_{I} )\right|_x = 0
	\end{equation*}
	by (\ref{eqprll}). Moreover, Lemma \ref{lmlccp} implies
	\begin{equation*}
		\langle \overline{\zeta_l}  \wedge \alpha , \overline{\zeta_j}  \wedge \beta \rangle_{\omega} (x) = \langle  \zeta_j ,  \zeta_l \rangle_{\omega} (x) \cdot \langle  \alpha ,  \beta \rangle_{\omega} (x).
	\end{equation*}
	Thus, combining the above argument, we get
	\begin{equation*}
		( \delbar_{\rho} \alpha , \delbar_{\rho} \beta )_{L^2} = \int_X \left( \sum_{j,l=1}^d c_l (\rho) \overline{c_j (\rho)} \langle  \zeta_j ,  \zeta_l \rangle_{\omega} \right) \langle  \alpha ,  \beta \rangle_{\omega}  \frac{\omega^n }{n!} .
	\end{equation*}
	Since $\zeta_1, \dots , \zeta_n$ are parallel with respect to $\omega$, its $\omega$-metric inner product $\langle  \zeta_j ,  \zeta_l \rangle_{\omega} (x)$ is a constant that does not depend on $x \in X$, and hence the summation inside the bracket is a constant over $X$. Thus we get, as required,
	\begin{equation*}
		(\delbar_{\rho} \alpha , \delbar_{\rho} \beta)_{L^2} = \mathrm{const.} ( \alpha ,  \beta )_{L^2} = 0
	\end{equation*}
	by the assumption that $\alpha$ is $L^2$-orthogonal to the holomorphic forms.
\end{proof}

Thus, from (\ref{eqiilbbeta}), (\ref{eqiilbalpha}), and Proposition \ref{prop_hashimoto1}, we get
\begin{align*}
\|\delbar_\rho u\|_{L^2}^2
&=\|\delbar_\rho \alpha + \delbar_\rho \beta\|_{L^2}^2\\
&=\|\delbar_\rho \alpha\|_{L^2}^2 + 2{\rm Re}\,(\delbar_\rho \alpha, \delbar_\rho \beta)_{L^2} + \|\delbar_\rho \beta\|_{L^2}^2\\
&\geq \frac{1}{2} \left( \lambda_{p,0} - C_3  |1 - \rho |^2 \right) \|\alpha\|_{L^2}^2 +  \frac{1}{{\rm Vol}_\omega(X)} |1-\rho|^2\cdot\|\beta\|_{L^2}^2 ,
\end{align*}
for any $u \in A^{p,0} (X)$, with $u = \alpha + \beta$, $\alpha\in (\mathbb{H}^{p, 0})^\perp$, $\beta\in \mathbb{H}^{p, 0}$.

We take $\rho$ to be sufficiently close to the identity so that
\begin{equation*}
	\frac{1}{2} \left( \frac{\lambda_{p,0}}{|1-\rho|^2} - C_3 \right) \ge  \frac{1}{{\rm Vol}_\omega(X)},
\end{equation*}
or equivalently
\begin{equation*}
	|1 - \rho |^2 \le \frac{\lambda_{p,0} {\rm Vol}_\omega(X)}{C_3{\rm Vol}_\omega(X) +2},
\end{equation*}
which is possible since $\lambda_{p,0}, C_3, {\rm Vol}_\omega(X) >0$ are constants that depend only on $X$, $\omega$, and the $L^2$-orthonormal basis $\zeta_1, \zeta_2, \dots, \zeta_d$ for $H^0(X, \Omega_X^1)$. Thus there exists an open neighbourhood $B$ around the identity in ${\rm Pic}^0(X)\cong \mathbb{C}^d/\Lambda_{\bf c}$, which is contained in the interior of the fundamental domain $D$, such that 
\begin{align*}
\|\delbar_\rho u\|_{L^2}^2 &\geq |1-\rho|^2  \left( \frac{1}{2} \left( \frac{\lambda_{p,0}}{|1-\rho|^2} - C_3 \right) \|\alpha\|_{L^2}^2 + \frac{1}{{\rm Vol}_\omega(X)} \cdot\|\beta\|_{L^2}^2 \right) \\
&\geq \frac{1}{{\rm Vol}_\omega(X)} |1-\rho|^2  \left(  \|\alpha\|_{L^2}^2 + \|\beta\|_{L^2}^2 \right) \\
&= \frac{1}{{\rm Vol}_\omega(X)} |1-\rho|^2 \cdot \| u \|_{L^2}^2
\end{align*}
holds for all $\rho \in B \setminus \{ 1 \}$ and all $u \in A^{p,0} (X)$, as required. This completes the proof of Theorem \ref{thm:main_2}.

\begin{remark}\label{rmk:prof_of_thm_main_2}
Note that the argument at the end of the proof of Theorem \ref{thm:main_1} is valid as long as the following condition holds: $H^{p, 0}(X, F)=0$ holds for any $F\in \mathcal{P}(X)\setminus \{\mathbb{I}_X\}$, since Proposition \ref{prop:usc} also holds for $(p, 0)$-forms under this condition, which can be shown by running the same argument as in the proof of this proposition just by replacing the equation (\ref{eq:expression_of_K_rho}) with 
\[
K_\rho = \sup\left\{\left.\frac{\|u\|_{L^2}}{\|\delbar_\rho u\|_{L^2}}\right|u\in A^{p, 0} (X) \setminus\{0\}\right\}. 
\]
Therefore, Theorem \ref{thm:main_2} remains true even after replacing $B$ with ${\rm Pic}^0(X)$, if $H^{p, 0}(X, F)=0$ holds for any $F\in {\rm Pic}^0(X)\setminus \{\mathbb{I}_X\}$. 
\end{remark}

\subsection{Proof of Corollary \ref{cor:main}}
Here we give an alternative proof of \cite[Lemma 4]{Ueda} when the manifold $X$ is K\"ahler (See also \cite[\S 8.3]{KoikeUehara} for an alternative proof of Ueda's lemma (with effective $K$), which the second author learned from Prof. Tetsuo Ueda). 

The assertion is clear if $F$ is holomorphically trivial. 
When $F\not=\mathbb{I}_X$, the corollary follows from 
Theorem \ref{thm:main_1}, Lemma \ref{lem:flat_H0=0}, and Lemma \ref{lem:CD-corresp}.



\section{Examples} \label{section:excex}

\subsection{Elliptic curves}

In this subsection, let us consider the case where $X$ is an elliptic curve $\mathbb{C}/\langle 1, \tau\rangle$, where $\tau$ is a complex number with ${\rm Im}\,\tau>0$. 
Note that, for any $F\in \mathcal{P}(X)\setminus\{\mathbb{I}_X\}$, any $v\in A^{0, 1}(X,F)$ admits a solution $u\in A^{0, 0}(X,F)$ such that $\delbar u=v$, since $H^{0, 1}(X, F)=0$ holds for any $F\in \mathcal{P}(X)\setminus\{\mathbb{I}_X\}$ by Lemma \ref{lem:flat_H0=0} and the Riemann--Roch theorem. 
In what follows, we denote by $z$ the standard coordinate of the universal covering $\mathbb{C}$ of $X$ and by $dz$ the non-trivial global holomorphic $1$-form on $X$ which comes from the $1$-form $dz$ on $\mathbb{C}$. 
Additionally, here we will use the Weierstrass $\sigma$-function 
\[
\sigma(z) := z\prod_{w\in \Lambda\setminus\{0\}}
\left(\left(1-\frac{z}{w}\right)\exp\left(\frac{z}{w}+\frac{z^2}{2w^2}\right)\right)
\]
and the Weierstrass $\zeta$-function
\[
\zeta(z) := \frac{1}{z} + \sum_{w\in \Lambda\setminus\{0\}}
\left(\frac{1}{z-w}+\frac{1}{w}+\frac{z}{w^2}\right)
\]
for the lattice $\Lambda:=\langle 1, \tau\rangle$. Recall the well-known formulae $\eta_1+\eta_2+\eta_3=0$, 
\[
\eta_3\omega_2-\eta_2\omega_3
=\eta_2\omega_1-\eta_1\omega_2
=\eta_1\omega_3-\eta_3\omega_1
=\frac{1}{2}\pi\sqrt{-1}, 
\]
and 
\[
\sigma(z+2\omega_j)=-e^{2\eta_j(z+\omega_j)}\sigma(z)
\]
for 
\[
\omega_1:=\frac{1}{2},\quad
\omega_2:=\frac{-1-\tau}{2}, \quad
\omega_3:=\frac{\tau}{2}
\]
and $\eta_j := \zeta(\omega_j)$ (See \cite[Chapter 23.2]{DLMF} for example). 

Suppose $F\in \mathcal{P}(X)\setminus\{\mathbb{I}_X\}$. 
Take an element $(p, q)\in [0, 1)^2\setminus\{(0, 0)\}$ such that 
\[
\rho(1)=e^{2\pi\sqrt{-1}p},\quad
\rho(\tau)=e^{2\pi\sqrt{-1}q}
\]
holds for the ${\rm U}(1)$-representation $\rho$ which corresponds to the flat line bundle $F$ via the correspondence in the proof of Proposition \ref{prop:alb}. 
Consider the meromorphic function $\widetilde{K}\colon \mathbb{C}^2\to \mathbb{P}^1$ defined by 
\[
\widetilde{K}(z, \xi) := e^{A(z-\xi)}\cdot \frac{\sigma(z-\xi+B)}{\sigma(z-\xi)}, 
\]
where we are letting $A:= -2(p\eta_3-q\eta_1)$ and $B:=p\tau-q$. 
As is easily shown by applying the formulae above, 
$\widetilde{K}$ induces an meromorphic section $K$ of the line bundle ${\rm Pr}_1^*F\otimes {\rm Pr}_2^*(F^{-1})$ on $X\times X$, where ${\rm Pr}_j\colon X\times X\to X$ denotes the $j$-th projection for $j=1, 2$. 
Using this $K$, one can concretely construct the solution of the $\delbar$-equation as follows.
\begin{proposition}\label{prop:kernel_sub}
Let $X=\mathbb{C}/\langle 1, \tau\rangle$, $F\in \mathcal{P}(X)\setminus\{\mathbb{I}_X\}$, and $(p, q)\in \mathbb{R}^2\setminus\{(0, 0)\}$ be as above. 
Then, for any $v\in A^{0, 1}(X,F)$, 
\[
u(z) := \frac{1}{\pi\sigma(p\tau-q)}\int_{\xi\in X}K(z, \xi)\cdot f(\xi)\,\frac{\sqrt{-1}}{2}d\xi\wedge d\overline{\xi}
\]
is an element of $A^{0, 0}(X,F)$ which 
satisfies $\delbar u=v$, where $f\in A^{0, 0}(X,F)$ is the element such that $v=fd\overline{z}$. 
\end{proposition}

\begin{proof}
Again by applying the formulae above, it follows that the meromorphic function 
$z\mapsto e^{Az}\cdot \frac{\sigma(z+B)}{\sigma(z)}$ on $\mathbb{C}$ induces a meromorphic section $k$ of $F$ on $X$. Note that $K(z, \xi)=k(z-\xi)$ holds. Note also that 
if follows from Poincar\'e--Lelong formula that
\[
\delbar \left(\frac{k(z)}{2\pi\sqrt{-1}}\,dz\right)=\sigma(B)\cdot \delta_0(z)
\]
holds as currents, where $\delta_0$ denotes the delta measure with support $0$. 
By using this equation (and using $kd\overline{z}/(2\pi\sqrt{-1}\sigma(B))$ instead of Bochner--Martinelli kernel), one can run the same argument as in \cite[\S 3. D.]{Demailly_book} to obtain a variant of Koppelman formula
\[
\int_{(z, \xi)\in X\times X}\delbar\left(K(z, \xi)(dz-d\xi)\right)\wedge \left(G(\xi)d\overline{\xi})\wedge (H(z)dz\right)
=-4\pi\sigma(B)\int_{z\in X}G(z)\cdot H(z)\,\frac{\sqrt{-1}}{2}dz\wedge d\overline{z}
\]
for any $G\in A^{0, 0}(X,F)$ and $H\in A^{0, 0}(F^{-1})$. 
From this formula, for arbitrary $\rho(z)dz\in A^{1, 0}(F^{-1})$, one can calculate the value of 
\[
I:=\int_{X\times X}d\left((K(z, \xi)\cdot(dz-d\xi))\wedge (f(\xi)d\overline{\xi})\wedge (\rho(z)dz)\right)
\]
as follows: 
\begin{align*}
I=& 
\int_{X\times X}\delbar(K(z, \xi)\cdot(dz-d\xi))\wedge (f(\xi)d\overline{\xi})\wedge (\rho(z)dz)\\
&+\int_{X\times X}(K(z, \xi)\cdot(dz-d\xi))\wedge (f(\xi)d\overline{\xi})\wedge \delbar(\rho(z)dz)\\
=& 
-4\pi\sigma(B)\int_{z\in X}f(z)\cdot \rho(z)\,\frac{\sqrt{-1}}{2}dz\wedge d\overline{z}\\
&+\int_{X\times X}(K(z, \xi)\cdot(dz-d\xi))\wedge (f(\xi)d\overline{\xi})\wedge \delbar(\rho(z)dz)\\
=& 
-4\pi\sigma(B)\int_{z\in X}f(z)\cdot \rho(z)\,\frac{\sqrt{-1}}{2}dz\wedge d\overline{z}-2\sqrt{-1}\int_X \delbar(u(z))\wedge \rho(z)dz\\
=& 
4\pi\sigma(B)\cdot \frac{\sqrt{-1}}{2}\int_{z\in X}(f(z)d\overline{z})\wedge \rho(z)dz-2\sqrt{-1}\int_X \delbar(u(z))\wedge \rho(z)dz.  
\end{align*}
On the other hand, as
\[
(K(z, \xi)\cdot(dz-d\xi))\wedge (f(\xi)d\overline{\xi})\wedge (\rho(z)dz)
=\left(-K(z, \xi)f(\xi)\rho(z)\right) d\xi\wedge d\overline{\xi}\wedge dz
\]
holds and $-K(z, \xi)f(\xi)\rho(z)$ is a function on $X\times X$, it follows from Stokes theorem that $I=0$. 
Therefore the assertion follows, since $\rho(z)dz\in A^{1, 0}(F^{-1})$ is arbitrary. 
\end{proof}

For $u$ and $v$ as in Proposition \ref{prop:kernel_sub}, it follows from Young's convolution inequality that 
\[
\|u\|_{L^2} \leq \left\|\frac{k}{\pi\sigma(p\tau-q)}\right\|_{L^1}\cdot \|v\|_{L^2}, 
\]
where $k$ is the meromorphic section of $F$ on $X$ as in the proof. 
Note that 
\[
\left\|\frac{k}{\pi\sigma(p\tau-q)}\right\|_{L^1}
=\frac{1}{\pi}\int_D\left|e^{-2(p\eta_3-q\eta_1)z}\cdot \frac{\sigma(z+p\tau-q)}{\sigma(p\tau-q)\sigma(z)}\right|\,\frac{\sqrt{-1}}{2}dz\wedge \overline{z}
\]
by definition of $k$, where $D$ is a fundamental domain of the covering map $\mathbb{C}\to X$. 
As is followed from a simple calculation that the function 
\[
p-q\tau \mapsto \left|e^{-2(p\eta_3-q\eta_1)z}\cdot \frac{\sigma(z+p\tau-q)}{\sigma(p\tau-q)\sigma(z)}\right|
\]
is also periodic with respect to the lattice $\langle 1, \tau\rangle$, it follows from the compactness of $X$ that there exists a positive constant $M$ such that 
\[
\left|e^{-2(p\eta_3-q\eta_1)z}\cdot \frac{\sigma(z+p\tau-q)}{\sigma(p\tau-q)\sigma(z)}\right|\leq \frac{M}{\mathsf{d}(z)\cdot \mathsf{d}(p\tau-q)}
\]
holds for any $z, p\tau-q\in \mathbb{C}$, where $\mathsf{d}$ denotes the Euclidean distance from the lattice $\langle 1, \tau\rangle$. 
Therefore, by letting $K:=M\|1/\mathsf{d}(z)\|_{L^1}$, 
one has the inequality 
\[
\|u\|_{L^2} \leq \frac{K}{\mathsf{d}(p\tau-q)}\cdot \|v\|_{L^2}
\]
for $u$ and $v$ as in Proposition \ref{prop:kernel_sub}, which proves Theorem \ref{thm:main_1} when $X$ is an elliptic curve. 

\subsection{Counterexamples to naive generalisations}

There are several non-trivial hypotheses in Theorem \ref{thm:main_2}, but all of them are necessary. We present below various counterexamples when we naively drop these hypotheses. Whether they are best possible is a question that we do not discuss in details, but relative simplicity of the examples below seems to suggest that they may not be too far from being optimal. 

\subsubsection{$(p, q)$-forms with $q\not=0$}
Let $u$ be a $(p, q)$-form on a K\"ahler manifold $(X, \omega)$. 
In the proof of Theorem \ref{thm:main_1} and \ref{thm:main_2}, we estimated the value $\|\delbar_\rho u\|_{L^2}$ by using $\|u\|_{L^2}$ from below for a ${\rm U}(1)$-representation $\rho$ of the fundamental group $\pi_1(X, *)$ when $q=0$. 
Here we give an example which implies that the assumption $q=0$ is essential. 

Let $X$ be a compact complex torus $\mathbb{C}^2/\Lambda$ of dimension $2$, where $\Lambda\subset \mathbb{C}^2$ is a lattice of rank $4$. 
We attach the Euclidean metric to $X$. 
Consider a $(1, 1)$-form $u:=dz_1\wedge d\overline{z_2}\in A^{1, 1}$ on $X$, where $(z_1, z_2)$ is the standard coordinate of $\mathbb{C}^2$. 
Then, for a ${\rm U}(1)$-representation $\rho$ of $\Lambda$ with $c_1(\rho) = 0$ and $c_2 (\rho) = \ve$ (which exists by Lemma \ref{lmctrisoqt}; it is close to, but does not equal, the trivial representation $1$ when $0 < \ve \ll 1$) and the corresponding perturbed $\delbar$-operator $\delbar_\rho:=\delbar + \ve d\overline{z_2}\wedge$, one has that $\delbar_\rho u = 0$, whereas $\|u\|_{L^2}\not=0$. 

\subsubsection{Restriction to the neighbourhood of the identity when $p>0$}\label{eg:p>0}

By Theorem \ref{thm:main_2}, there exists a neighbourhood $B$ of $\mathbb{I}_X$ in $\mathcal{P}(X)$ such that $H^{p, 0}(X, F)=0$ holds for any $F\in B\setminus \{\mathbb{I}_X\}$. 
Here let us give some examples concerning the following question: how large can $B$ be? 
As we know the answer when $p=0$ by Lemma \ref{lem:flat_H0=0},  we will consider the case of $p>0$ in what follows. 

First, let $X$ be a compact complex torus. 
Then, as the cotangent bundle $\Omega_X^1$ is the holomorphically trivial vector bundle of rank ${\rm dim}\,X$, 
\[
H^{p, 0}(X, F)=H^0(X, \Omega_X^1\otimes F)
=\bigoplus_{{\rm dim}\,X}H^0(X, F)
\]
holds. Therefore it follows from Lemma \ref{lem:flat_H0=0} that $H^{p, 0}(X, F)=0$ holds for any $F\in {\rm Pic}^0(X)\setminus \{\mathbb{I}_X\}$, which means that one can take $B=\mathcal{P}(X)$ when $X$ is a compact complex torus (Note that $\mathcal{P}(X)={\rm Pic}^0(X)$ in this case). 

Next, let $X$ be a hyperelliptic surface, i.e.~$X$ is the quotient of a product of two elliptic curves by a finite abelian group. It is well-known that there exists a positive integer $m$ such that $K_X^{\otimes m}=\mathbb{I}_X$ holds, whereas $K_X$ is not holomorphically trivial. 
Therefore, for a flat line bundle $F:=K_X^{-1}\in\mathcal{P}(X)\setminus\{\mathbb{I}_X\}$, it holds that 
\[
H^{2, 0}(X, F) = H^0(X, K_X\otimes F) = H^0(X, \mathbb{I}_X) \cong\mathbb{C}, 
\]
which implies that one {\it cannot} take $B=\mathcal{P}(X)$ for $p=2$ in this case. 

From these two examples, it seems natural to ask whether one can take $B={\rm Pic}^0(X)$ in general, whereas the latter example means that one cannot further enlarge $B$ to be $\mathcal{P}(X)$. 

\begin{question}
Does Theorem \ref{thm:main_2} hold even after replacing $B$ with ${\rm Pic}^0(X)$? 
\end{question}

\subsubsection{Non-vanishing Ricci curvature}

Let $X$ be a compact Riemann surface of genus $g(X) \ge 2$. For any non-trivial flat line bundle $F$, $F^{-1}$ is also non-trivial and flat and hence we get
\begin{equation*}
	\dim H^0 (X,F \otimes K_X) = \dim H^0 (X,F^{-1}) -\deg (F^{-1}) +g(X)-1 = g(X)-1 \ge 1
\end{equation*}
by Riemann--Roch and Lemma \ref{lem:flat_H0=0}. In particular, $H^{1,0} (X,F) \neq 0$ holds for any non-trivial flat line bundle $F$, meaning that Theorem \ref{thm:main_2} cannot be generalised (at least naively) to the case when the Ricci curvature does not vanish.

\bibliography{ueda.bib}


\newpage 

\appendix

\section{Another proof of (\ref{eq:bddness_of_supK(F)_on_Pict}) by constructing finite coverings}

Here we give another proof of (\ref{eq:bddness_of_supK(F)_on_Pict}), the boundedness of $\{K(F)\mid F\in {\rm Pic}^t(X)\}$ for a non-trivial element $t\in G$, where $K(F)$ is as in Remark \ref{rmk:final_rmk_thm_1_1_i}. 
Note that we may (and thus we will) assume $d:={\rm dim}\,H^1(X, \mathcal{O}_X)>0$, since otherwise it is trivial as $\mathcal{P}(X)\setminus {\rm Pic}^0(X)$ is a finite set. 

As $G$ is a finite group, there exists a positive integer $m$ such that $F^m:=F^{\otimes m}$ is topologically trivial for any $F\in {\rm Pic}^t(X)$. 
Note that $m>1$ since $t\not=0$. 
Consider the map 
\[
{\rm Pic}^t(X)\ni F\mapsto F^m\in {\rm Pic}^0(X). 
\]
As it is a local homeomorphism from a compact space to a connected Hausdorff space, it is surjective. 
Thus there exists an element $F_0\in {\rm Pic}^t(X)$ such that $F_t^m\cong \mathbb{I}_X$. 
Note that such an element is not unique. 
Indeed, as $d>0$, one can take an element $L\in {\rm Pic}^0(X)\setminus \{\mathbb{I}_X\}$ such that $L^m\cong \mathbb{I}_X$. Another element which enjoys this property can be constructed as $F_1:=F_0\otimes L$. 
The assertion follows from the following: 
\begin{lemma}\label{lem:key_lem__in_prf_of_thm_1_1_ii}
Let $\mu\in \{0, 1\}$. 
Then the function $K(\bullet)$ is bounded on ${\rm Pic}^t(X)\setminus U_\mu$ for any small open neighbourhood $U_\mu$ of $F_\mu$. 
\end{lemma}

As is clear, it follows from Lemma \ref{lem:key_lem__in_prf_of_thm_1_1_ii} that $K(\bullet)$ is a bounded function on ${\rm Pic}^t(X)$. 
Therefore it has the supremum $K_0:=\sup\{K(F)\mid F\in {\rm Pic}^t(X)\}<\infty$, from which the assertion follows. 

\begin{proof}[Proof of Lemma \ref{lem:key_lem__in_prf_of_thm_1_1_ii}]
Consider the finite covering $\pi\colon\widetilde{X}\to X$ of $X$ constructed as 
\[
\widetilde{X} := \{\xi\in F_\mu\mid \xi^{\otimes m}=1\}, 
\]
where $1$ is a fixed global holomorphic frame of the holomorphically trivial line bundle $F_\mu^m$ and $\pi$ is the restriction of the projection $F_\mu\to X$. 
Then the line bundle $\pi^*F_\mu$ is holomorphically trivial, since one can construct a global holomorphic frame tautologically by construction, i.e. 
\[
\{(\xi, \xi)\mid \xi\in\widetilde{X}\}\subset \{(\xi, \eta)\in \widetilde{X}\times F_\mu\mid \eta\in F_\mu|_{\pi(\xi)}\}=\pi^*F_\mu
\]
defines a global frame of $\pi^*F_\mu$. 
Therefore, as any element of ${\rm Pic}^t(X)$ is topologically isomorphic to $F_\mu$, one can consider the map 
\[
\pi^*|_{{\rm Pic}^t(X)}\colon {\rm Pic}^t(X)\ni F\mapsto  \pi^*F\in {\rm Pic}^0(\widetilde{X}). 
\]
For proving the assertion, it is sufficient to show that $\pi^*|_{{\rm Pic}^t(X)}$ is injective. 
Indeed, if it holds, then the image of ${\rm Pic}^t(X)\setminus U_\mu$ by the map $\pi^*|_{{\rm Pic}^t(X)}$ is contained in ${\rm Pic}^0(\widetilde{X})\setminus \widetilde{U}_\mu$ for some open neighbourhood $\widetilde{U}_\mu$ of $\mathbb{I}_{\widetilde{X}}$. 
Take the maximum $K_\mu := \{\widetilde{K}(F)\mid F\in {\rm Pic}^0(\widetilde{X})\setminus \widetilde{U}_\mu\}$, 
where $\widetilde{K}(\bullet)$ is the function on $\mathcal{P}(\widetilde{X})$ whose definition is the same as that of $K(\bullet)$, i.e.~$\widetilde{K}(F)$ the minimum of all the positive constants $M$ such that the condition {\bf ($L^2$-estimate)$_M$} holds for $(\widetilde{X}, F, \pi^*\omega, h)$, where $h$ is a flat metric on $F$ (the existence of the maximum is a consequence of Proposition \ref{prop:usc}). 
By Lemma \ref{lem:lem_for_key_lem__in_prf_of_thm_1_1_ii} below, it follows that $K(F)\leq K_\mu$ for any $F\in {\rm Pic}^t(X)\setminus U_\mu$, which is nothing but the assertion. 

In what follows, we show that $\pi^*|_{{\rm Pic}^t(X)}$ is injective. 
Take $F, G\in {\rm Pic}^t(X)$ such that $\pi^*F\cong \pi^*G$. 
Set $\rho:=m_X(F\otimes G^{-1})$, i.e.~$\rho\colon \pi_1(X, *)\to {\rm U}(1)$ is the representation which corresponds to the flat bundle $F\otimes G^{-1}$. 
As $F, G\in {\rm Pic}^t(X)$, $F\otimes G^{-1}$ is topologically trivial. 
Thus, from the isomorphism (\ref{eq:isom_m_X_Pic0}), it follows that $\rho\in {\rm Hom}_{\rm group}(H_1 (X, \mathbb{Z})_{\rm free}, {\rm U(1)})$. 

As $\pi$ is a covering, one has the natural exact sequence 
\[
1 \to \pi_1(\widetilde{X}, *) \to \pi_1(X, *) \to \Gamma \to 1
\]
of groups, where $\Gamma$ is the deck transformation group of $\pi$. 
As the abelianisation functor is right exact, one obtains the exact sequence 
\[
H_1(\widetilde{X}, \mathbb{Z}) \to H_1(X, \mathbb{Z}) \to \Gamma_{\rm ab} \to 0, 
\]
where $\Gamma_{\rm ab}$ is the abelianisation of $\Gamma$, which induces the exact sequence 
\[
0\to {\rm Hom}_{\rm group}(\Gamma_{\rm ab}, {\rm U}(1))
\to {\rm Hom}_{\rm group}( H_1(X, \mathbb{Z}), {\rm U}(1))
\to {\rm Hom}_{\rm group}(H_1(\widetilde{X}, \mathbb{Z}), {\rm U}(1)). 
\]
From $\pi^*(F\otimes G^{-1})\cong \mathbb{I}_{\widetilde{X}}$, 
it follows that $\rho$ is an element of the kernel of the last arrow in the sequence above. 
Thus, by the exactness, it tunes out that $\rho$ comes from an element of ${\rm Hom}_{\rm group}(\Gamma_{\rm ab}, {\rm U}(1))$, or equivalently, 
there exists a group homomorphism $f\colon \Gamma_{\rm ab}\to {\rm U}(1)$ such that the following diagram commutes. 
\begin{displaymath}
		\xymatrixcolsep{8pc}\xymatrixrowsep{4pc}\xymatrix{ 
H_1(X, \mathbb{Z}) \ar@{->}[r]^-{\rho} \ar@{->>}[d] &  {\rm U}(1)  \\
\Gamma_{\rm ab} \ar@{->}[ur]_-{f}& } 
\end{displaymath}
Recall that $\rho\in {\rm Hom}_{\rm group}(H_1 (X, \mathbb{Z})_{\rm free}, {\rm U(1)})$. 
Therefore $\rho([\gamma])=1$ holds for any torsion element $[\gamma]\in \pi_1(X, \mathbb{Z})$. 
As $\pi$ is a finite covering, $\Gamma$ is a finite group. 
Thus one has that $f\equiv 1$, from which $F\cong G$ follows. 
\end{proof}

\begin{lemma}\label{lem:lem_for_key_lem__in_prf_of_thm_1_1_ii}
Let $M$ be a positive number, $F$ an element of $\mathcal{P}(X)$, $h$ a flat metric on $F$, and $\pi\colon \widetilde{X}\to X$ be as in the proof of Lemma \ref{lem:key_lem__in_prf_of_thm_1_1_ii}. 
Assume that $\pi^*F$ is not holomorphically trivial. 
Then, if the condition {\bf ($L^2$-estimate)$_M$} holds for $(\widetilde{X}, \pi^*F, \pi^*\omega, \pi^*h)$, the condition {\bf ($L^2$-estimate)$_M$} holds also for $(X, F, \omega, h)$. 
\end{lemma}

\begin{proof}
Take a smooth $\delbar$-closed $(0, 1)$-form $v$ with values in $F$ whose Dolbeault cohomology class $[v]\in H^{0, 1}(X, F)$ is trivial. 
Then clearly $\pi^*v$ is a $(0, 1)$-form with values in $\pi^*F$ whose Dolbeault cohomology class $[\pi^*v]\in H^{0, 1}(\widetilde{X}, \pi^*F)$ is trivial. 
Thus, if the condition {\bf ($L^2$-estimate)$_M$} holds for $(\widetilde{X}, \pi^*F, \pi^*\omega, \pi^*h)$, 
there exists a smooth global section $\widetilde{u}$ of $\pi^*F$ such that $\delbar \widetilde{u}=\pi^*v$ and 
\[
\sqrt{\int_{\widetilde{X}}|\widetilde{u}|_{\pi^*h}^2\,dV_{\pi^*\omega}} \leq M \sqrt{\int_{\widetilde{X}}|\pi^*v|_{\pi^*h, \pi^*\omega}^2\,dV_{\pi^*\omega}}
=M \sqrt{\ell\cdot \int_{X}|v|_{h, \omega}^2\,dV_{\omega}}
\]
hold, where $\ell$ is the degree of the covering map $\pi$. 
Let $\gamma$ be a deck transformation of $\pi$. 
Then $\gamma^*\widetilde{u}$ is also a solution of the $\delbar$-equation, since 
$\delbar\gamma^*\widetilde{u}=\gamma^*\delbar\widetilde{u}=\gamma^*\pi^*v=\pi^*v$. 
Therefore it follows from Lemma \ref{lem:flat_H0=0} and the assumption that $\widetilde{u}$ is invariant by the action of any deck transformation of $\pi$. 
Thus there exists a smooth global section $u$ of $F$ such that $\pi^*u=\widetilde{u}$, which proves the assertion since
\[
\sqrt{\int_{\widetilde{X}}|\widetilde{u}|_{\pi^*h}^2\,dV_{\pi^*\omega}} 
=\sqrt{\int_{\widetilde{X}}|\pi^*u|_{\pi^*h}^2\,dV_{\pi^*\omega}} 
=\sqrt{\ell\cdot \int_{X}|u|_{h}^2\,dV_{\omega}} 
\]
holds. 
\end{proof}

\end{document}